%% file: main.tex
\documentclass[a4paper]{amsart}
\input{abbr}
\usepackage[utf8]{inputenc}
\usepackage{amsmath,amssymb,amsthm}
\usepackage{mathrsfs}
\usepackage[british]{babel}
\usepackage{chemarrow}
\usepackage{eucal}
\usepackage{graphicx}
\usepackage{color}
\usepackage{url}
\usepackage{esint}
\usepackage{tikz}
\usepackage{mathtools}
\usepackage{stmaryrd}
\usetikzlibrary{arrows,shapes}
\usetikzlibrary{fit}
\usetikzlibrary{calc,3d}
\usepackage{lipsum}  
\usepackage{tikz-3dplot}
\usepackage{subcaption}
\usepackage{makecell} 
 \usepackage{booktabs}
 \usepackage{array}
\usepackage{caption}
\usepackage[sort,compress]{natbib}
\usepackage{siunitx}
\usepackage[all,cmtip]{xy}
\usetikzlibrary{patterns, arrows.meta, positioning, decorations.markings}
\tikzset{
    >=stealth',
    punkt/.style={
           rectangle,
           rounded corners,
           draw=black, very thick,
           text width=6.5em,
           minimum height=2em,
           text centered},
    pil/.style={
           ->,
           thick,
           shorten <=2pt,
           shorten >=2pt,}
}

\newcommand{\rot}{\operatorname{rot}}

 \newtheorem{theorem}{Theorem}
 \newtheorem{lemma}[theorem]{Lemma}

 \theoremstyle{definition}

 \theoremstyle{remark}
 \newtheorem{remark}[theorem]{Remark}
 
\numberwithin{equation}{section}
\numberwithin{theorem}{section}

\begin{document}
\title[{Nonconforming Curved Elements for Stokes Equation}]{
Pressure-Robust Fortin--Soulie Elements of the Stokes Equation on Curved Domains
}

\author{Wei Chen}
\address{LMAM and School of Mathematical Sciences, Peking University, Beijing 100871, P. R. China.}
\address{Chongqing Research Institute of Big Data, Peking University, Chongqing 401329, P. R. China}
\email{2406397052@pku.edu.cn}

\author{Zhen Liu*}
\address{Institut für Mathematik, Friedrich-Schiller-Universität Jena, 07743 Jena, Germany}
\address{Chongqing Research Institute of Big Data, Peking University, Chongqing 401329, P. R. China}
\email{z.liu@uni-jena.de}

\thanks{* Corresponding author.}

\thanks{\textbf{Funding:} The first author acknowledges support by the Postdoctoral Science Foundation of China Grants No. 2025M783065. The second author acknowledges support by the Sino-German (CSC-DAAD) Postdoc Scholarship Program, ID 202406010516.}

\date{}

\begin{abstract}
This paper presents a pressure-robust and element-wise divergence-free nonconforming finite element method for the Stokes problem on curved domains. 
The discrete element is constructed by mapping the Fortin--Soulie element from a reference triangle using an isoparametric mapping for the geometry and a Piola transform for the function space. The inf-sup condition and the error estimate with optimal convergence rate are proved. Pressure-robustness is obtained by replacing the discrete velocity test functions with the first-order Raviart--Thomas functions. Numerical examples are provided to validate the theoretical results.
\end{abstract}

 

\maketitle


\section{Introduction}

Let $\Omega\subset \mathbb{R}^2$ be a bounded domain with a possibly curved boundary $\partial \Omega$. This paper considers the Stokes equation
\begin{equation}
\label{eq::Stokes}
\begin{aligned}
-\nu \Delta \bu + \nabla p & = \bft && \text{ in } \Omega, & \\
\Div \bu & = 0 && \text{ in } \Omega,  &\\
\bu&=0 &&\text{ on } \partial\Omega, &
\end{aligned}    
\end{equation}
where $\bft$ denotes the body force, $\nu> 0$ the kinematic
viscosity, $\bu$ the fluid velocity, and $p$ the pressure. 
In numerical simulations of the Stokes equation~\eqref{eq::Stokes}, it is desirable for discrete schemes to be divergence-free $\Div\bu_h=0$, where $\bu_h$ is the discrete velocity.
On polygonal domains, this property ensures pressure-robustness by decoupling the velocity error from the pressure, which is advantageous for flows with large pressure gradients or small viscosity.
Since the construction of the divergence-free conforming finite elements is not trivial, numerous techniques have been proposed, which refer to~\cite{scott1985norm,guzman2018inf,fabien2022low,zhang2005a,christiansen2018generalized,beirao2017divergence,guzman2014conforming2D} and the references therein.
Alternatively, nonconforming elements~\cite{coruzeix1873conforming,fortin1983non,fortin1985a,crouzeix1989nonconforming,baran2007gauss} can achieve element-wise divergence-free ($\Div_h\bu_h=0$), while they typically require velocity reconstruction operators to recover pressure-robustness~\cite{linke2012a,linke2014on,linke2016robust,apel2022a}.
Despite their effectiveness on polygonal domains, these methods typically suffer from a drop to second-order accuracy on curved domains due to boundary approximation errors.
To recover higher-order accuracy, isoparametric finite element methods~\cite{lenoir1986optimal,ciarlet1972interpolation, bernardi1989optimal} are commonly employed. Nevertheless, a direct application of standard isoparametric techniques fails to preserve both the divergence-free and pressure-robust properties~\cite{bae2004finite,dione2015penalty}.

Recently, Neilan and Otus~\cite{neilan2021divergence} developed the first exactly divergence-free method on curved domains.
Their construction treats the lowest-order Scott--Vogelius pair as finite element spaces defined on macro-elements to ensure stability and utilizes the Piola transform to maintain the divergence-free property. Pressure-robustness is obtained by constructing a special approximation of the source term through commuting projections. 
This approach has been further extended to arbitrary polynomial orders~\cite{durst2024a} and $H^1$-conforming approximations~\cite{neilan2023a}. A boundary correction method based on the Clough--Tocher refinements has also been proposed~\cite{liu2023a}.

In \cite{fortin1983non}, Fortin and Soulie developed a stable, element-wise divergence-free element by enriching the quadratic Lagrange element with some nonconforming bubble functions. 
This paper is devoted to extending such elements to curved domains, ensuring element-wise divergence-free velocity, stability, optimal convergence, and pressure-robustness.
Following \cite{neilan2021divergence}, the Piola transform is employed to obtain the element-wise divergence-free velocity. Note that the degrees of freedom for the Fortin--Soulie element, including its extension to curved domains, cannot be directly defined in the sense of Ciarlet \cite[Definition 3.1.1]{brenner2008the}. 
Consequently, the core of the construction lies in the definition of degrees of freedom for the mapped quadratic Lagrange space $\bWh$ and the mapped nonconforming bubble space $\bPhi_h$, respectively.
The global velocity space is defined as $\bVh=\bWh+\bPhi_h$  
and is paired with the pressure space $Q_h$ to solve the Stokes equation~\eqref{eq::Stokes}.
The inf-sup condition follows by a two-step method~\cite{stenberg1984analysis,hu2015finite,hu2015a}: the piecewise constant component of $Q_h$ is controlled by $\mathbf{W}_h$, while its orthogonal component is controlled by $\mathbf{\Phi}_h$.
However, as a result of the construction, $\bWh$ is continuous only at the endpoints and midpoint, and $\bPhi_h$ is continuous solely at the two Gauss--Legendre points for the internal edge shared by elements with curved edges. Consequently, the velocity space $\bVh$ lacks continuity at any point across such edges. This nonconformity introduces two issues: the estimation of consistency errors and the preservation of pressure-robustness.

To facilitate the consistency error analysis, an operator $\bE_h: \bVh \to \bVh^{iso}$ is defined via separate interpolations of the component spaces $\bWh$ and $\bPhi_h$ into their isoparametric counterparts. Here, $\bVh^{iso}$ is the standard second-order nonconforming isoparametric Fortin--Soulie space. 
The proof of the approximation properties for $\bE_h$ relies on a stable decomposition of $\bVh$ presented herein. 
By leveraging these approximation properties in conjunction with the weak continuity property of $\bVh^{iso}$, the optimal velocity error estimate is established:
\begin{equation*}
\begin{aligned}
\| \nabla \bu- \nabla_h \bu_h \|_{L^2(\Omega_h)}\leq& Ch^2 \big(\| \bu \|_{H^3(\Omega)} + \nu^{-1}\|p\|_{H^2(\Omega)} + \nu^{-1}h\|\bft\|_{H^3(\Omega)}\big).
\end{aligned}    
\end{equation*}
The presence of $\nu^{-1}$ indicates a persistent coupling between the velocity error and the pressure scale, implying that a direct application of the pair $(\bVh, Q_h)$ fails to be pressure-robust.
To address this, a modified scheme is proposed by replacing the standard Galerkin source term $\int_{\Omega_h}\bft_h\cdot\bv_h\dx$ with the variational crime $\int_{\Omega_h}\bft_h\cdot\Pi_h\bv_h\dx$, where $\Pi_h$ denotes the velocity reconstruction operator from $\bVh$ to the first-order parmetric Raviart--Thomas spaces~\cite{bertrand2014first}.
While such reconstruction techniques have been constructed for nonconforming elements on polygon domains~\cite{linke2012a,linke2014on,linke2016robust}, the current modification extends this framework to curved domains.
This modification, combined with the computable approximation $\bft_h$ from~\cite[Section~6]{neilan2021divergence}, leads a pressure-independent velocity error estimate:
$$
\|\nabla\bu-\nabla_h\bu_h^R\|_{L^2(\Omega_h)}\leq Ch^2\|\bu\|_{H^5(\Omega)},
$$
where $\bu_h^R$ is the discrete velocity solution of the modified scheme.
To the best of the authors' knowledge, this work represents the first extension of the reconstruction-based framework to curved domains.

The direct generalization of this method to three dimensions is, however, non-trivial. A primary obstacle is the absence of an analogous quadrature rule in higher dimensions~\cite{outs2004divergence}. The integration errors prevent the method from reaching its theoretical optimal convergence. 
In \cite{li2025divergence}, Li et al. developed an exactly divergence-free method for three-dimensional curved domains by employing high-order parametric Brezzi--Douglas--Marini elements and interior-penalty discontinuous Galerkin techniques.
Although the present study is confined to the two-dimensional case, these developments provide key insights into the potential generalization of the three-dimensional nonconforming Fortin element \cite{fortin1985a} to curved domains.

The rest of the paper is organized as follows:  Section~\ref{sec::preliminaries} introduces notations and preliminary results required for the subsequent analysis. In section~\ref{set::curved_fortin_soulie_element}, the nonconforming curved Fortin--Soulie element is constructed, followed by the verification of the discrete inf-sup stability. Section~\ref{sec::error_estimate} carries out a detailed error analysis, while the development of a pressure-robust formulation is detailed in section~\ref{sec::pressure_robust_scheme}. Section~\ref{sec::numerical_examples} provides numerical experiments that confirm the theoretical results. Some auxiliary results are given in Appendix~\ref{sec::appendix}.


\section{Preliminaries}
\label{sec::preliminaries}

This section introduces the notations and the weak formulation. Some results of the transformation of curved domains and the affine Fortin--Soulie element are provided for constructing and analyzing the curved Fortin--Soulie element.

\subsection{Notations}
\label{sec:notation}
Given a bounded domain $G$ with the boundary $\partial G$, the usual Sobolev spaces $W^{m,p}(G)$ with norm $\| \bullet\|_{W^{m,p}(G)}$ and semi-norm $| \bullet |_{W^{m,p}(G)}$ are used~\cite{adams2003sobolev}. Let $H^m(G) = W^{m,2}(G)$ and $L^p(G) = W^{0,p}(G)$. Furthermore, $H_0^1(G)$ is the subspace of $H^1(G)$ consisting of functions with vanishing traces on $\partial G$, and $L^2_0(G)$ denotes the space of $L^2(G)$ functions with zero mean over $G$. Denote the space of $m$-times continuously differentiable functions on $G$ by $C^m(G)$ and the space of all polynomials on $G$ with degree less than or equal to $k$ by $P_k(G)$. Corresponding vector-valued functions and spaces are denoted by boldface letters, e.g., $\bv \in \bP_k(G)$. Let $\bH(\Div;G) \coloneqq \{\bv\in\bL^2(G):~\Div \bv \in L^2(G)\}$.

Let $\hT$ be the reference triangle with vertices $(0,0)^\intercal,(1,0)^\intercal$, and $(0,1)^\intercal$. Denote by  $\hNT\coloneqq\{\hat{a}_i\}_{i=1}^6$ the set of quadratic Lagrange nodes and by $\hGT\coloneqq\{\hat{g}_i\}_{i=1}^6$ the set of
Gauss--Legendre points~\cite{fortin1983non} as shown in Figure~\ref{fig:points_hT}. Define
\begin{equation}
\label{def:hbPhi}
\hbPhi(\hT) \coloneqq \operatorname{span}\left\{
 (\phi_{\hT}(\hat{x}),0)^\intercal, \, 
 (0, \phi_{\hT}(\hat{x}) )^\intercal \right\},
\end{equation}
where $\phi_{\hT}(\hat{x})=2-3[(1-\hat{x}_{1}-\hat{x}_2)^2+\hat{x}_1^2+\hat{x}_2^2]$ with $\hat{x}=(\hat{x}_1,\hat{x}_2)^\intercal$. Since $\phi_{\hT}(\hat{x})$ vanishes at all points in $\hGT$, $\bPhi(\hT)$ is a nonconforming quadratic Gauss--Legendre bubble function space in $\bP_2(\hT)$.

Throughout this paper, the constant $C$ will denote a generic positive constant which is independent of the mesh size. For a regular mapping $H:~\mathbb{R}^2 \rightarrow \mathbb{R}^2$, denote its Jacobian by $DH$.

\begin{figure}[h!]
    \centering
    \tikzset{
        axis/.style={->, >=latex, gray, thin},
        tri/.style={thick, black},
        node_a/.style={circle, fill=black, inner sep=0pt, minimum size=4pt},
        node_g/.style={circle, draw=black, line width=0.8pt, fill=white, inner sep=0pt, minimum size=4pt}
    }
    \begin{subfigure}[b]{0.48\textwidth}
        \centering
        \begin{tikzpicture}[scale=2.5]
            \draw[tri] (0,0) -- (1,0) -- (0,1) -- cycle;
            \coordinate (a1) at (0,0);
            \coordinate (a2) at (1,0);
            \coordinate (a3) at (0,1);
            \coordinate (a4) at (0.5,0);
            \coordinate (a5) at (0.5,0.5);
            \coordinate (a6) at (0,0.5);
            \foreach \i in {1,...,6} {
                \node[node_a] at (a\i) {};
            }
            \node[below left]  at (a1) {$\hat{a}_1$};
            \node[below right] at (a2) {$\hat{a}_2$};
            \node[above left]  at (a3) {$\hat{a}_3$};
            \node[below]       at (a4) {$\hat{a}_6$};
            \node[above right] at (a5) {$\hat{a}_4$};
            \node[left]        at (a6) {$\hat{a}_5$};
        \end{tikzpicture}
        \caption{Quadratic Lagrange nodes.}
        \label{fig:Lagrange_points}
    \end{subfigure}
    \hfill
    \begin{subfigure}[b]{0.48\textwidth}
        \centering
        \begin{tikzpicture}[scale=2.5]
            \draw[tri] (0,0) -- (1,0) -- (0,1) -- cycle;
            \def\gL{0.2113} 
            \def\gH{0.7887} 
            \coordinate (g1) at (\gL, 0);
            \coordinate (g2) at (\gH, 0);
            \coordinate (g3) at ({1-\gL}, \gL); 
            \coordinate (g4) at ({1-\gH}, \gH); 
            \coordinate (g5) at (0, \gH);
            \coordinate (g6) at (0, \gL);
            \foreach \i in {1,...,6} {
                \node[node_g] at (g\i) {};
            }
            \node[below]       at (g1) {$\hat{g}_5$};
            \node[below]       at (g2) {$\hat{g}_6$};
            \node[above right]       at (g3) {$\hat{g}_1$};
            \node[above right] at (g4) {$\hat{g}_2$};
            \node[left]        at (g5) {$\hat{g}_3$};
            \node[left]        at (g6) {$\hat{g}_4$};
        \end{tikzpicture}
        \caption{Gauss--Legendre points.}
        \label{fig:GaussLegendre_points}
    \end{subfigure}
    \caption{The quadratic Lagrange nodes and Gauss--Legendre points on the reference triangle $\hT$.}
    \label{fig:points_hT}
\end{figure}
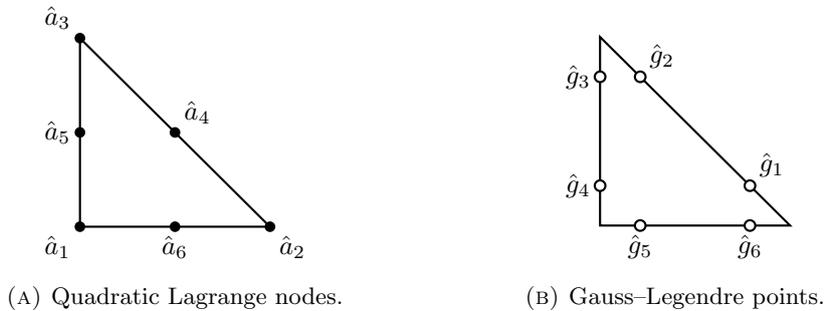

\subsection{Weak formulation}

Let $\bV =\bH_0^1(\Omega)$ and $Q = L_0^2(\Omega)$. The weak formulation of the Stokes equation \eqref{eq::Stokes} reads: Given $\bft \in \boldsymbol{L}^2(\Omega)$, find $(\bu,p) \in \bV \times Q$ such that
\begin{equation}
\label{var: continuous}
\begin{aligned}
\nu \int_{\Omega} \nabla \bu : \nabla \bv \dx -\int_\Omega p \Div \bv \dx & = \int_\Omega \bft \cdot  \bv  \dx  &&\text{ for all } \bv \in \bV, &\\
-\int_\Omega q \Div \bu  \dx &= 0 && \text{ for all } q \in Q.&
\end{aligned}
\end{equation}
Problem \eqref{var: continuous} is well-posed due to the inf–sup condition holding, see~\cite{girault2012finite}. Introduce the space of the divergence-free functions
$ \bZ \coloneqq \{ \bv \in \bV:~ \Div \bv =0 \text{ on } \Omega \}.$
The velocity $\bu$ can be equivalently characterized as: Find $\bu \in \bZ$ such that
\begin{equation}
\label{eqn:u}
\nu \int_{\Omega} \nabla \bu : \nabla \bv \dx  = \int_\Omega \bft \cdot  \bv  \dx \quad  \text{ for all } \bv \in \bZ.    
\end{equation}
Let $\mathrm{I}$ be the $2\times 2$ identity matrix. By introducing the “stress” tensor $\bsig = \nu \nabla \bu - p \mathrm{I}$ as the constitutive law, the equilibrium equation is governed by $\Div \bsig + \bft = 0$.
If $\bft = \nabla \psi$ with $\psi \in H^1(\Omega)$, setting $\bv = \bu \in \bZ$ in \eqref{eqn:u} gives
\begin{equation*}
\nu \int_{\Omega} \nabla \bu:~\nabla \bu \dx =  \int_{\Omega} \nabla \psi \cdot \bu \dx = - \int_{\Omega} \psi \Div \bu \dx = 0,
\end{equation*}
which implies $\bu = \mathbf{0}$. 
This means that an irrotational volume force is completely absorbed by the pressure.

\subsection{Curved triangulations}
\label{Sec::DefFT} 

In this paper, let $\Omega \subset \mathbb{R}^2$ be a sufficiently smooth domain with the curved boundary $\partial\Omega$. Following the standard isoparametric framework~\cite{lenoir1986optimal,ciarlet1972interpolation, bernardi1989optimal}, a curved triangulation of $\Omega$ is constructed from a shape-regular affine triangulation $\tcalTh$ whose boundary vertices lie on $\partial \Omega$. The associated polygonal domain $\tOmega_h\coloneqq{\rm int} \left(\cup_{\tT \in \tcalTh} \overline{\tT}\right)$ provides an $\mathcal{O}(h^2)$ approximation to $\Omega$, where $h$ is the maximum diameter of the triangle $\tT \in \tcalTh$. Furthermore, it is assumed that each $\tT\in\tcalTh$ has at most two boundary vertices.

Let $G:~\tOmega_h \to \Omega$ be a bijective map such that $G|_{\tT}(\bx) = x$ at all vertices of $\tT$ and $\|G\|_{W^{1,\infty}(\tOmega_h)}\le C$. Note that $G$ is the identity map for any triangle $\tT \in \tcalTh$ with three interior vertices.  Let $G_h$ be the piecewise quadratic nodal interpolant of $G$, which satisfies $\|DG_h\|_{W^{1,\infty}(\tT)}\le C$ and $\|DG_h^{-1}\|_{W^{1,\infty}(\tT)}\le C$ for all $\tT\in \tcalTh$. Define the isoparametric triangulation as $\calTh = \{G_h(\tT):~ \tT\in \tcalTh\}$ and the computational domain are defined as $\Omega_h\coloneqq{\rm int}\Big(\cup_{ T\in  \calTh} \overline{ T}\Big)$. 

Let the gradient operator $\nabla_h$ and the divergence operator $\Div_h$ be understood piecewise with respect to the triangulation $\calTh$.
Let $\mathcal{E}_h$ denote the set of all edges of $\calTh$.
Let $\bn$ be the outward unit normal of a domain, which is clear from its context. 
Notations marked with a tilde $\tilde{\bullet}$ refer to quantities associated with the triangulation $\tcalTh$.

For $\tT \in \tcalTh$, let $F_{\tT}:~\hT\to \tT$ be an affine mapping. Define the quadratic diffeomorphism $F_T:~\hT\to T$
as $F_T = G_h\circ F_{\tT}$ such that $F_T = F_{\tT}$ at the vertices of $\hT$. Note that, if $e\subset \partial T$ is a straight edge with $e = F_{T}(\hat e)$ and $\hat e\subset \partial \hT$,
then $F_T|_{\hat e}$ is affine.  If $T\in \calTh$ has all straight edges, then $F_T$ is affine and $T = G_h(\tT) = \tT$.
Moreover, given $T \in \tcalTh$ with $h_T = \text{diam}(G_h^{-1}(T))$, it holds that~\cite{neilan2021divergence}
\begin{subequations}
\label{Prop-FT}
\begin{align}
& |F_T|_{W^{m,\infty}(\hat{T})} \le C h_T^m \quad 0\le m\le 2, \quad |F_T^{-1}|_{W^{m,\infty}(T)} \le C h_T^{-m} \quad 0\le m\le 3, \label{Prop-FT:eq1}\\
& C h_T^2 \le \det(D F_T) \text{ and } \det(D F_T) \le C h_T^2. \label{Prop-FT:eq2}
\end{align}
\end{subequations}
Recall that $\hNT\coloneqq\{\hat{a}_i\}_{i=1}^6$ is the set of Lagrange nodes in $\hT$. Let $\tNT =F_{\tT}(\hNT)$ and $\NT = F_T(\hNT)$ be the sets of Lagrange nodes on $\tT$ and $T$, respectively. The symbols $\tGT$ and $\GT$ are defined in a similar way but related to the Gauss--Legendre points on $\tT$ and $T$, respectively. 

For each $T\in \calTh$, define the contravariant Piola
mapping for the transformation of vector fields $A_T:~\hT\to \mathbb{R}^{2\times 2}$, see, e.g.,~\cite[Sect. 2.1.3]{boffi2013mixed}, as 
\begin{equation}
\label{def: piola-AT}
A_T(\hat{x}) = DF_T(\hat{x}) / \det(DF_T(\hat{x})).    
\end{equation}
Well-known properties of the Piola transform~\cite[(2.1.71)]{boffi2013mixed} show
\begin{equation}
\label{pro-AT:div}
\Div\bv(x)=\frac{1}{\det(DF_T(\hat{x}))}\Hat{\Div}\hbv(\hat{x}) \text{ with } \bv(\bx)=A_T(\hat{x})\hbv(\hat{x}) \text{ and } x = F_T(\hx).
\end{equation}
It follows from~\cite[Lemma 2.3]{neilan2021divergence} that
\begin{align}
\label{ATBounds}
|A_T|_{W^{m,\infty}(\hT)}\le C h_T^{m-1} \text{ and }
|A_T^{-1}|_{W^{m,\infty}(\hT)}
\le \left\{
\begin{array}{ll}
Ch_T^{1+m} & m=0,1,\\
0 & m\ge 2.
\end{array}
\right.
\end{align}

The following lemmas will be used later.
\begin{lemma}[{\cite[Lemma 2.1.6]{boffi2013mixed}}]
\label{lem:pro-AT}
Let $\bv(\bx)=A_T(\hat{x})\hbv(\hat{x})$ and $q(x) = \hq(\hx)$ with $x = F_T(\hx)$. Then it holds that
$$
\begin{aligned}
\int_{T} q \Div \bv  \dx &= \int_{\hT} \hat{\Div} \hq \hbv  \hdx,    \\
\int_{\partial T} \bv \cdot \bn_T \, q \ds &= \int_{\partial \hT}  \hbv \cdot \hbn_{\hT} \, \hq \hds.    \\
\end{aligned}
$$
\end{lemma}

\begin{lemma}[\cite{bernardi1989optimal,neilan2021divergence}]
\label{lem::NormInThT}
Suppose $\bv(\bx) = \hbv(\hat{x})$ with $x = F_T(\hx)$ for sufficiently smooth $\bv\in W^{m,p}(T)$.
Then it holds that 
\begin{align*}
|\bv|_{W^{m,p}(T)}&\le C h_T^{2/p-m} \sum_{r=0}^m h_T^{2(m-r)} |\hbv|_{W^{r,p}(\hT)},\\
|\hbv|_{W^{m,p}(\hT)} &\le C h_T^{m-2/p} \sum_{r=0}^m |\bv|_{W^{r,p}(T)},
\end{align*}
for any $T$ with $\hT = F_T^{-1}(T)$.
\end{lemma}

The following lemma describes the norm relationship of functions before and
after the Piola transform, which can be obtained by~\eqref{ATBounds} and Lemma~\ref{lem::NormInThT}.
\begin{lemma}
\label{lem::normH1XThXT}
Suppose $\bv(\bx)=A_T(\hat{x})\hbv(\hat{x})$ with $x = F_T(\hx)$ for sufficiently smooth $\bv\in W^{m,p}(T)$. Then it holds that
\begin{align*}
|\bv|_{W^{m,p}(T)}&\le C h_T^{2/p-m-1} \sum_{r=0}^m  h_T^{m-r}|\hbv|_{W^{r,p}(\hT)},\\
|\hbv|_{W^{m,p}(\hT)} &\le C h_T^{m+1-2/p} \sum_{r=0}^m |\bv|_{W^{r,p}(T)},
\end{align*}
for any $T$ with $\hT = F_T^{-1}(T)$.
\end{lemma}

\begin{proof}
The first inequality in Lemma \ref{lem::NormInThT}, Leibniz's rule, and \eqref{ATBounds} give
\begin{equation*}
\begin{aligned}
|\bv|_{W^{m,p}(T)}&\le C h_T^{2/p-m} \sum_{r=0}^m h_T^{2(m-r)} |A_T \hbv|_{W^{r,p}(\hT)} \\
& \le C h_T^{2/p-m} \sum_{r=0}^m h_T^{2(m-r)}  \sum_{j=0}^r  |A_T|_{W^{r-j,\infty}(\hT)} |\hbv|_{W^{j,p}(\hT)} \\
& \le C h_T^{2/p-m-1} \sum_{j=0}^m \sum_{r=j}^m h_T^{2m-r-j} |\hbv|_{W^{j,p}(\hT)},
\end{aligned}   
\end{equation*}
which proves the first inequality. The second inequality follows from similar procedures. 
\end{proof}

\subsection{Affine Fortin--Soulie element}
Define $\hbW(\hT)\coloneqq \bP_2(\hT)$ and recall the space $\hbPhi(\hT)$ defined on the reference triangle $\hT$ in \eqref{def:hbPhi}. The corresponding spaces $\tbW(\tT)$ and $\tbPhi(\tT)$ on the affine triangle $\tT \in \tcalTh$ with $F_{\tT}:~\hT \rightarrow \tT$ are defined as
\begin{equation*}
\begin{aligned}
\tbW(\tT) & \coloneqq \{ \tbw_T:~\tbw_T(\tbx) = \hbw_T(\hat{x}) \text{ with } \hbw_T \in \hbW(\hT) \text{ and } \tbx = F_{\tT}(\hat{x})\}, \\  
\tbPhi(\tT) &\coloneqq  \{\tbphi_T:~\tbphi_T(\tbx)=\hbphi_T(\hat{x})\text{ with } \hbphi_T \in \hbPhi(\hT) \text{ and } \tbx = F_{\tT}(\hat{x})\}.    
\end{aligned}
\end{equation*}
The affine Fortin--Soulie space~\cite{fortin1983non} on $\tcalTh$ is defined as 
$$
\begin{aligned}
\tbV_h \coloneqq & \{\tbv_h:~ \tbv_h|_{\tT}\in\tbW(\tT) \text{ for all } \tT\in\tcalTh, \tbv_h \text{ is continuous (zero) at the two}  \\
& \text{Gauss--Legendre points of each internal (boundary) edge of } \tcalTh\},
\end{aligned}
$$
which is a nonconforming piecewise quadratic approximation of $\bH^1_0(\tOmega_h)$. It follows from~\cite[Proposition 1]{fortin1983non} that $\tbV_h = \tbW_h \oplus \tbPhi_h$, where
\begin{equation*}
\begin{aligned}
\tbW_h & \coloneqq \{\tbw_h \in \bH^1_0(\tOmega_h):~\tbw_h|_{\tT}\in\tbW(\tT)\ \text{ for all } \tT \in \tcalTh \},\\  
\tbPhi_h & \coloneqq \{\tbphi_h \in \bL^2(\tOmega_h):~\tbphi_h|_{\tT}\in\tbPhi(\tT)\ \text{ for all } \tT\in\tcalTh\},
\end{aligned}    
\end{equation*}
are the standard conforming quadratic Lagrange space of $\bH^1_0(\tOmega_h)$ and the nonconforming quadratic Gauss--Legendre bubble function space, respectively. 

Introduce  the global piecewise linear space on the affine triangulation $\tcalTh$ with zero integration value as
$$
\tQ_h \coloneqq \{\tq_h \in L^2_0(\tOmega_h):~ \tq_h |_{\tT}\in\tQ(\tT) \text{ for all } \tT\in\tcalTh\},
$$
with $\tQ(\tT) \coloneqq \{ \tq_{\tT} :~ \tq_{\tT}(\tbx) = \hq_{\hT}(\hat{x}) \text{ with } \hq_{\hT} \in \hQ(\hT) \text{ and } \tbx = F_{\tT}(\hat{x})\}$ and $\hQ(\hT) \coloneqq P_1(\hT)$.

The affine Fortin--Soulie method employs the pair $(\tbV_h, \tQ_h)$ for solving the Stokes problem on the affine triangulation $\tcalTh$. The method is element-wise divergence-free since $\tDiv_h \tbV_h \subset \tQ_h$ and well-posedness since $(\tbV_h, \tQ_h)$ satisfies the inf-sup condition. Readers can refer to~\cite{fortin1983non} for detailed proof. 

Recall that the isoparametric mapping $G_h:~\tOmega_h\to\Omega_h$. Given $\tT\in\tcalTh$ and $T\in\calTh$ with $T=\Gh(\tT)$, define the local operator $\boldsymbol{\Upsilon}_T:~\bH^2(\tT)\to \bH^2(T)$ as
\begin{equation}
\label{def:bUpsilon}
(\boldsymbol{\Upsilon}_T\tbv)(\bx) = \tbv(\tbx)\, \text{ for all } \tbx \in\tT \text{ with } x=\Gh(\tbx).    
\end{equation}
The global operator $\boldsymbol{\Upsilon}_h$ can be given by $\boldsymbol{\Upsilon}_h|_T = \boldsymbol{\Upsilon}_T$ for all $T \in \calTh$.
Define the isoparametric velocity space on $\Omega_h$ as
\begin{equation}
\label{def::Vhiso}
\bVh^{iso}\coloneqq\{\bv_h \in \bL^2(\Omega_h):~\bv_h =\boldsymbol{\Upsilon}_h \tbv_h \text{ with }  \tbv_h \in\tbV_h\}.
\end{equation}
Since the functions in $\bVISO$ are continuous at the Gauss--Legendre points of all interior edges and vanish at the Gauss--Legendre points of all boundary edges, $\bVISO$ is a nonconforming finite element approximation of $\bH^1_0(\Omega_h)$. Note that $\bVISO$, paired with $Q_h$ in \eqref{defQinterpolate}, cannot yield a divergence-free solution of the Stokes problem~\cite{neilan2021divergence}. In the following, the curved Fortin--Soulie element space $\bVh$ will be obtained by mapping $\hbW(\hT)$ and $\hbPhi(\hT)$ via an isoparametric mapping of the geometry and a Piola transform of the function space.

\section{Curved Fortin--Soulie Element}
\label{set::curved_fortin_soulie_element}

This section constructs the nonconforming Fortin--Soulie element on curved domains. The finite element method for the Stokes problem is formulated, and its well-posedness is verified by proving a global discrete inf-sup condition using a two-step technique.

\subsection{Curved Fortin--Soulie element}
Recall the definition of the Piola transform $A_T$ in \eqref{def: piola-AT}. Define the spaces on $T \in \calTh$ with $F_T:~\hT \to T$ by the Piola transform as 
\begin{equation}
\label{def:WandPhi}
\begin{aligned}
\bW(T) &\coloneqq \{\bw_T:~ \bw_T(\bx) = A_T(\hat{x}) \hbw_T(\hat{x}) \text{ with } \hbw_T \in \hbW(\hT) \text{ and }x = F_{T}(\hat{x}) \}, \\
\bPhi(T) &\coloneqq \{\bphi_T:~ \bphi_T(\bx) = A_T(\hat{x}) \hbphi_T(\hat{x}) \text{ with } \hbphi_T \in \hbPhi(\hT) \text{ and }x = F_{T}(\hat{x}) \}.     
\end{aligned}
\end{equation}
Note that $\bW(T) = \tbW(\tilde T)$ if $F_T$ is an affine mapping. For any $\bw_T \in \bW(T)$ and a straight edge $e\subset \partial T$, $\bW(T)$ is not necessarily a piecewise polynomial space, and the restriction of $\bw_T$ to $e$ is not necessarily a polynomial even though $F^{-1}_T$ is affine on $e$. Nonetheless, the normal component $\bw_T \cdot \bn_e|_e$ with $\bn_e$ the unit normal of the straight edge $e$ is a quadratic polynomial~\cite[Lemma 3.1]{neilan2021divergence}. Similar results hold for functions $\bphi_T \in \bPhi(T)$. The local space $\bW(T)$ has the following property. 

\begin{lemma}
\label{lem:uniqueXT}
A function $\bw_T \in\bW(T)$ is uniquely determined by its values $\bw_T(\ba)$ with $\ba\in \NT = \{ \ba_i\}_{i=1}^6$. Given $\bv \in \bH^3(T)$, define the nodal interpolation operator $\bIT:~\bH^3(T)\to\bW(T)$ by $(\bIT\bv)(\ba)=\bv(\ba)$ for all $\ba\in\NT$. It holds that
\begin{equation*}
\|\bv -\bIT \bv \|_{H^m(T)}\leq Ch_T^{3-m}\|\bv\|_{H^3(T)} \quad m=0,1 \text{ for all } \bv \in\bH^3(T).    
\end{equation*}
\end{lemma}

\begin{proof}
The number of degrees of freedom equals the dimension of $\bW(T)$. Given any $\bw_T(\bx) = A_T(\hat{x}) \hbw_T (\hat{x})$ for some $\hbw_T(\hat{x})\in\hbW(\hT)$ with $x = F_T(\hx)$ and $\bw_T(\ba)=0$ for $\ba\in \NT$. Since $A_T(\hat{x})$ is invertible, it holds that
$$
\hbw_T(\hat{a}) = A^{-1}_T(\hat{a}) \bw_T(\ba) = 0 \text{ for all } \hat{a} \in \hNT = \{ \hat{a}_i\}_{i=1}^6.
$$
This implies $\bw_T=0$ since these values uniquely determine $\hbw_T$.
The approximation property is established by following an argument analogous to that presented in~\cite[Lemma 3.5]{neilan2021divergence}.
\end{proof}

Define the global spaces on $\calTh$ as 
\begin{equation*}
\begin{aligned}
\bWh  &\coloneqq  \{ \bw_h:~\bw_h|_T \in \bW(T) \text{ for all } T \in \calTh, \bw_h \text{ is continuous (zero) on the two}  \\
& \text{endpoints and the midpoint of each internal (boundary) edge of } \calTh \},  \\
\bPhi_h &\coloneqq \{ \bphi_h \in \bL^2(\Omega_h):~ \bphi_h|_T \in \bPhi(T) \text{ for all } T \in \calTh  \}.  
\end{aligned}    
\end{equation*}
Following similar procedures as in~\cite[Theorem 4.2]{neilan2021divergence}, one can obtain 
\begin{equation}
\label{pro-Wh:space}
\bWh \subset \bH_0(\Div;\Omega_h)\coloneqq \{\bv\in\bL^2(\Omega_h):~ \Div \bv \in L^2(\Omega_h),~\bv\cdot\bn|_{\partial\Omega_h}=0\}.    
\end{equation}
Based on these two global spaces, define the discrete velocity space on $\calTh$ as 
\begin{equation}
\label{def::bVh}
\bVh\coloneqq \{\bv_h \in \bL^2(\Omega_h):~\bv_h=\bw_h+\bphi_h \text{ with } \bw_h \in\bWh \text{ and } \bphi_h \in\bPhi_h\}.
\end{equation}

\begin{remark}
\label{remark:Vh}
For any internal edge $e$ shared by two affine elements, $\bw_h \in \bWh$ is continuous on $e$, while $\bphi_h \in \bPhi_h$ satisfies continuity only at the two Gauss–Legendre points. Thus, the velocity $\bv_h = \bw_h + \bphi_h \in \bVh$ is continuous at these two points. For any internal edge $e$ shared by elements with curved edges, $\bw_h \in \bWh$ is continuous solely at the endpoints and the midpoint of $e$. This discrepancy implies that the velocity space $\bVh$ fails to satisfy continuity at any point along $e$.
\end{remark}

The following lemma establishes a stable decomposition of $\bVh$ on local elements, whose proof is technical and given in Appendix \ref{app::Prooflem:VTestimate1}.
This result facilitates the analysis of some properties of $\bVh$ by decomposing the space into $\bWh$ and $\bPhi_h$.

\begin{lemma}
\label{lem:VTestimate1}
Let $\bv_T = \bw_T + \bphi_T$ with $\bw_T \in\bW(T)$ and $\bphi_T\in\bPhi(T)$. If there exists a point $\ba\in\NT$ such that $\bw_T(\ba)=0$, then it holds that
\begin{align}
\label{bvcbvbL2Bd} \|\bw_T\|_{L^2(T)}+\|\bphi_T\|_{L^2(T)}&\leq C\|\bv_T\|_{L^2(T)}.
 \end{align}
 Moreover, if $\bw_T$ vanishes on the two endpoints and the midpoint of an edge $e$ of $T$, then $\bw_T|_e = 0$, and for $h_T$ small enough, it holds that 
 \begin{align}\label{bvL2ConByNabv}
 \|\bv_T\|_{L^2(T)}&\leq Ch_T\|\nabla\bv_T\|_{L^2(T)}.
\end{align}
\end{lemma}

By Lemma~\ref{lem:VTestimate1}, the following theorem shows that $\bVh$ is the direct sum of the $\bH_0(\Div;\Omega_h)$ conforming space $\bWh$ and the nonconforming bubble space $\bPhi_h$.

\begin{theorem}
\label{the::VhEqXhDirSumPhih}
It holds that $\bVh = \bWh \oplus \bPhi_h$.     
\end{theorem}
\begin{proof}
Given $\bv_h \in\bWh\cap\bPhi_h$ and let $\bv_T \coloneqq \bv_h|_T$ for any $T \in \calTh$. For $T \in \calTh$, if there is a boundary edge $e \subset \partial T \cap \partial \Omega_h$, it follows from the definition of $\bWh$ and Lemma~\ref{lem:VTestimate1} that $\bv_T|_e = 0$. This gives $\bv_T\equiv0$ since $\bv_h \in\bPhi_h$ implies $\bv_T$ also belonging to $\bPhi(T)$. Otherwise, there exists a sequence of elements $T_l, T_{l-1}, \cdots, T_1, T_0$ such that $T_l=T$, consecutive elements share a common edge, and there exists a boundary edge $e \subset \partial T_0 \cap \partial \Omega_h$. Denote the edge $e^\prime = \partial T_1 \cap \partial T_0$. Note that $\bv_{T_0}\equiv0$ implies $\bv_{T_0}|_{e^\prime}=0$. This, the definition of $\bWh$, and Lemma \ref{lem:VTestimate1} show that $\bv_{T_1}|_{e^\prime}=0$. Again, this gives $\bv_{T_1} \equiv0$ since $\bv_h \in\bPhi_h$ implies $\bv_{T_1}$ also belonging to $\bPhi(T_1)$. Proceeding inductively across all elements gives $\bv_{T_l} \equiv0$ on the element $T_l$, which completes the proof. 
\end{proof}

 \subsection{The connection between $\bVh$ and $\bVh^{iso}$.}

According to Remark \ref{remark:Vh}, functions in $\bVh$ lack continuity across certain internal edges. Based on the direct sum decomposition of $\bVh$ in Theorem~\ref{the::VhEqXhDirSumPhih}, this subsection establishes a connection between $\bVh$ and the standard isoparametric space $\bVh^{iso}$, which serves as the foundation for the weak continuity analysis in Section \ref{sec::weak_continuity}.

To begin with, following~\cite[Section 3.2]{neilan2021divergence}, the connection between subspaces of $\bVh$ and $\tbV_h$ is constructed as follows. Let $\tT\in\tcalTh$ and $T\in\calTh$ with $T=\Gh(\tT)$. Introduce the local operator $\Psi^W_T:~\tbW(\tT)\to\bW(T)$ uniquely by the conditions
\begin{equation*}
(\Psi^W_T\tbw_T)(a)=\tbw_T(\tba) \text{ for all } \tba\in\tNT \text{ and } a=\Gh(\tba).    
\end{equation*} 
Then the global operator $\Psi^W_h$ can be given by $\Psi^W_h|_T = \Psi^W_T$ for all $T \in \calTh$. Define the local operator $\Psi^\Phi_T:~\tbPhi(\tT)\to\bPhi(T)$ uniquely by the condition
\begin{equation*}
(\Psi^\Phi_T \tbphi_T)(b)=\tbphi_T(\tilde{b}) \text{ for } \tilde{b} \text{ is the barycenter of } \tT \text{ and } b=\Gh(\tilde{b}).   
\end{equation*}
Then the global operator $\Psi_h^\Phi$ can be given by $\Psi_h^\Phi|_T = \Psi_T^\Phi$ for all $T \in \calTh$. As a result, one can obtain
\begin{equation}
\label{eq:relations}
\begin{aligned}
\bWh & = \{\bw_h :~\bw_h=\Psi_h^W \tbw_h \text{ with } \tbw_h \in\tbW_h\},  \\
\bPhi_h & = \{\bphi_h:~\bphi_h = \Psi_h^\Phi \tbphi_h \text{ with }  \tbphi_h \in\tbPhi_h\}.
\end{aligned}   
\end{equation}
The following theorem provides properties of the local operator $\Psi^W_T$, which can be obtained by combining Lemma \ref{lem:uniqueXT} and~\cite[Theorem 3.7]{neilan2021divergence}.

\begin{theorem}
\label{Psitbvandbv}
The following properties of the local operator $\Psi^W_T$ hold:
\begin{enumerate}
    \item[(a)] If $F_T$ is affine, then $(\Psi^W_T\tbv)(\bx)=\tbv(\tbx)$ and $\Psi^W_T$ is the identity operator.
    \item[(b)] If $e\subset\partial T$ is a straight edge, so that $e\subset\partial\tT$, then
    $(\Psi^W_T\tbv)\cdot\bn_e|_e=\tbv\cdot\bn_e|_e$.
    \item[(c)] There holds $\|\Psi^W_T\tbv\|_{H^1(T)}\leq C\|\tbv\|_{H^1(\tT)}.$
\end{enumerate}
\end{theorem}

Based on the global operators $\Psi_h^W$, $\Psi_h^\Phi$, and the relation in \eqref{eq:relations}, define the operator $\bE^W_h:~\bWh \to \bVh^{iso}$ as $
(\bE^W_h \bw_h)(\bx) = (\Psi_h^W)^{-1} \bw_h(\tbx)$ with $\bx = G_h(\tbx)$, and the operator $\bE^\Phi_h:~\bPhi_h \to \bVh^{iso}$ as $(\bE^\Phi_h \bphi_h)(\bx) = (\Psi_h^\Phi)^{-1} \bphi_h(\tbx)$ with $\bx = G_h(\tbx)$. Then, for any $\bv_h = \bw_h + \bphi_h \in \bVh$ with $\bw_h \in \bWh$ and $\bphi_h \in \bPhi_h$, 
define the operator $\bE_h:~\bVh \to \bVh^{iso}$ as 
\begin{equation}
\label{def:Eh}
\bE_h \bv_h = \bE^W_h \bw_h + \bE^\Phi_h \bphi_h. 
\end{equation}
By Theorem \ref{the::VhEqXhDirSumPhih}, the operator $\bE_h$ is well-defined since $\bw_h$ and $\bphi_h$ are unique. 
The approximation property of $\bE_h$ can be derived as follows.

\begin{lemma}
\label{lem::EhError}
For all $\bv_h \in\bVh$ and $T\in\calTh$, it holds that
\begin{equation}
\label{ErrorEh}
\|\bv_h -\Eh\bv_h\|_{L^2(T)}+h_T\|\nabla(\bv_h-\Eh\bv_h)\|_{L^2(T)}\leq Ch_T^2\|\nabla\bv_h\|_{L^2(T)}.
\end{equation}
\end{lemma}

\begin{proof}
A triangle inequality gives
\begin{equation}
\label{TriEhbv}
\|\bv_h -\Eh \bv_h \|_{H^m(T)}\leq \|\bw_h-\bE^W_h \bw_h\|_{H^m(T)}+\|\bphi_h-\bE^\Phi_h \bphi_h\|_{H^m(T)}~~ m=0,1.    
\end{equation}
For $T\in\calTh$ an affine (noncurved) triangle, $\bE^W_h\bw_h$ and $\bE^\Phi_h\bphi_h$ are piecewise quadratic polynomials, and 
\eqref{ErrorEh} holds trivially in this case. For $T\in\calTh$ with curved boundary, following similar procedures as in~\cite[Lemma 4.5]{neilan2021divergence}, one can obtain 
\begin{equation}\label{InterrATVC}
\|\bw_h-\bE^W_h\bw_h\|_{H^m(T)} \leq C h_T^{1-m}\|\bw_h \|_{L^2(T)} ~~ m=0,1.    
\end{equation}
By definitions, there exist two functions $\hbphi_{\hT}$ and $\hbphi^{\prime}_{\hT}$ in $\hbPhi(\hT)$ such that  $\bphi_h|_T(\bx) = A_T(\hat{x})\hbphi_{\hT}(\hat{x})$ and $(\bE^\Phi_h\bphi_h)|_T(\bx) =\hbphi^{\prime}_{\hT}(\hat{x})$ with $x=F_T(\hx)$, respectively. Lemma \ref{lem::NormInThT} gives
\begin{equation}
\label{InterrATVB0}
\|\bphi_h-\bE^\Phi_h\bphi_h\|_{H^m(T)}\leq Ch_T^{1-m}\|A_T\hbphi_{\hT}-\hbphi_{\hT}^{\prime}\|_{H^m(\hT)}.
\end{equation}
It follows from the definition of $\bE^\Phi_h$ that
$A_T(\hat{b})\hbphi_{\hT}(\hat{b})=\hbphi^\prime_{\hT}(\hat{b})$, where $\hat{b}$ is the barycenter of~$\hT$. This combined with the definition of $\hbPhi(\hT)$ shows $\hbphi^\prime_{\hT}(\hat{x})=A_T(\hat{b})\hbphi_{\hT}(\hat{x})$ and 
\begin{equation}\label{ATvbEst}
|A_T\hbphi_{\hT}-\hbphi^\prime_{\hT}|_{H^m(\hT)}\leq \sum_{r=0}^m|A_T(\hat{x})-A_T(\hat{b})|_{W^{r,\infty}(\hT)}|\hbphi_{\hT}|_{H^{m-r}(\hT)}.    
\end{equation}
The Taylor expansion with $\hat{\xi} \in\hT$ gives 
$
|A_T(\hat{x})-A_T(\hat{b})| = |DA_T(\hat{\xi})(\hat{x}-\hat{b})|\leq C|A_T|_{W^{1,\infty}(\hT)}|\hat{x}-\hat{b}|.
$
The estimate for the norm of $A_T$ in \eqref{ATBounds} shows $|A_T(\hat{x})-A_T(\hat{b})|_{W^{r,\infty}(\hT)}\leq C$ for~$r=0,1$. Substituting this into \eqref{ATvbEst} and employing the equivalence of norms in a finite-dimensional setting show
$
|A_T\hbphi_{\hT}-\hbphi^\prime_{\hT}|_{H^m(\hT)}\leq C\|\hbphi_{\hT}\|_{H^m(\hT)}\leq C\|\hbphi_{\hT}\|_{L^2(\hT)}.
$
This, \eqref{InterrATVB0}, and Lemma~\ref{lem::normH1XThXT} lead to
\begin{equation}\label{InterrATVB}
\begin{aligned}
\|\bphi_h-\bE^\Phi_h\bphi_h\|_{H^m(T)}\leq C h_T^{1-m}\|\hbphi_{\hT}\|_{L^2(\hT)}\leq Ch_T^{1-m}\|\bphi_h\|_{L^2(T)} ~~m=0,1.
\end{aligned}
\end{equation}
Since $\bw_h|_T$ vanishes on $\partial T\cap\partial\Omega_h$, the proof of \eqref{ErrorEh} is completed by integrating Lemma \ref{lem:VTestimate1} with \eqref{TriEhbv}, \eqref{InterrATVC} and~\eqref{InterrATVB}.
\end{proof}

\begin{lemma}\label{col::H1Norm}
It holds that $\|\bv_h \|_{L^2(\Omega_h)}\leq C\|\nabla_h \bv_h\|_{L^2(\Omega_h)}$ for any $\bv_h \in \bVISO$ or $\bv_h \in\bVh$. 
\end{lemma}
\begin{proof}
Since any functions in $\bVISO$ are continuous at Gauss--Legendre points on all interior edges and vanish at Gauss--Legendre points on all boundary edges, discrete Poincar\'e--Friedrich's inequalities follow from the analysis in~\cite{brenner2003poincare}, as shown below
$$
\|\bv_h\|_{L^2(\Omega_h)}\leq C\|\nabla_h\bv_h\|_{L^2(\Omega_h)}  \text{ for all } \bv_h\in\bVISO.
$$
For any  $\bv_h \in \bVh$, a triangle inequality, this and Lemma \ref{lem::EhError} show
$$
\begin{aligned}
\|\bv_h\|_{L^2(\Omega_h)}&\leq \|\bv_h-\Eh\bv_h\|_{L^2(\Omega_h)}+\|\Eh\bv_h\|_{L^2(\Omega_h)}\\
&\leq \|\bv_h-\Eh\bv_h\|_{L^2(\Omega_h)}+C\|\nabla_h(\Eh\bv_h)\|_{L^2(\Omega_h)}\\
&\leq Ch^2\|\nabla_h\bv_h\|_{L^2(\Omega_h)}+C(\|\nabla_h(\Eh\bv_h-\bv_h)\|_{L^2(\Omega_h)}+\|\nabla_h\bv_h\|_{L^2(\Omega_h)})\\
&\leq C(1+h+h^2)\|\nabla_h\bv_h\|_{L^2(\Omega_h)},
\end{aligned}
$$
which completes the proof.   
\end{proof}

\subsection{The finite element method.}

Define the local space of pressure on each element $T \in \calTh$ as
\begin{equation}
\label{defQTloc}
\begin{aligned}
Q(T) & \coloneqq \{q_T:~ q_T(\bx) = {\hq_{\hT}(\hat{x})} \text{ with } \hq_{\hT} \in \hat Q(\hT) \text{ and }x = F_{T}(\hat{x}) \},   
\end{aligned}     
\end{equation}
and the global space  on $\calTh$ as
\begin{equation}
\label{defQinterpolate}
Q_h  \coloneqq \{q_h \in L^2(\Omega_h):~ q_h=\Upsilon_h \tq_h \text{ with } \tq_h \in\tQ_h\},    
\end{equation}
where $\Upsilon_h$ denotes the scalar counterpart of $\boldsymbol{\Upsilon}_h$ introduced in \eqref{def:bUpsilon}.

\begin{remark}
\label{re:qh}
Note that $\int_{\Omega_h} q_h \dx \neq 0$ for any $q_h \in Q_h$. In fact, by~\cite[Theorem 4.2]{neilan2023a}, $q_h \in Q_h$ satisfies
\begin{equation}
\label{re-eq:qh}
\begin{aligned}
& \sum_{T\in\calTh} \int_T\frac{2q_h|\tT|}{\det(DF_T\circ F_T^{-1})}\dx = \sum_{T\in \calTh} 2|\tT| \int_{\hT} q_h \circ F_{T} \rd \hat{x} \\
&\qquad  = \sum_{\tT \in \tcalTh} 2|\tT| \int_{\hT} \tq_h \circ F_{\tT} \rd \hat{x} = \sum_{\tT \in \tcalTh} \int_{\tT} \tq_h \rd \tbx= \int_{\tOmega_h} \tq_h \rd \tbx =0.    
\end{aligned}    
\end{equation}
\end{remark}

Recall the Stokes equation \eqref{eq::Stokes} defined on $\Omega$. 
Assume that $\partial \Omega$ and ${\bft}$ are sufficiently smooth
such that $(\bu,p)\in \bH^3(\Omega)\times H^2(\Omega)$,
and can be extended to $\mathbb{R}^2$ in a way such that
$(\bu,p)\in \bH^3(\mathbb{R}^2)\times H^2(\mathbb{R}^2)$ with $\Div \bu=0$ and
\begin{equation}
\label{extenUandP}
\|\bu\|_{H^3(\mathbb{R}^2)}\le C \|\bu\|_{H^3(\Omega)},\qquad \|p\|_{H^2(\mathbb{R}^2)}\le C \|p\|_{H^2(\Omega)}.    
\end{equation}
Then, extend ${\bft}$ by ${\bft} = -\nu \Delta \bu+ \nabla p$,
so that ${\bft}\in \bH^1(\mathbb{R}^2)$. Denote by ${\bft}_h\in \bL^2(\Omega_h)$ a computable approximation of ${\bft}|_{\Omega}$.

 Then the finite element method reads: Find $(\bu_h,p_h)\in \bVh\times Q_h$ such that
\begin{subequations}
\label{eqn:FEM}
\begin{alignat}{2}
\label{eqn:FEM1}
\int_{\Omega_h} \nu \nabla_h \bu_h: \nabla_h \bv_h \rd x - \int_{\Omega_h} p_h  \Div_h \bv_h \rd x & = \int_{\Omega_h} {\bft}_h\cdot \bv_h \rd x  &,\\
\label{eqn:FEM2}
-\int_{\Omega_h} q_h \Div_h \bu_h  \rd x & =0 &
\end{alignat}
\end{subequations}
for all $(\bv_h,q_h) \in \bVh \times Q_h$.

\subsection{Inf-sup stability.}
This subsection establishes the well-posedness of~\eqref{eqn:FEM} by using a two-step method~\cite{stenberg1984analysis,hu2015finite,hu2015a} to prove the inf-sup condition of the finite element pair $\bVh\times Q_h$.

Introduce two subspaces of $\hQ(\hT)$ and $Q(T)$ as
\begin{equation}
\label{defQ0T}
\begin{aligned}
\hQ_0(\hT) & \coloneqq \{ \hat{q}_{\hT}:~ \hat{q}_{\hT} \in \hQ(\hT) \text{ with } \int_{\hT} \hat{q}_{\hT} \mathrm{d} \hat{x} =0\}, \\
Q_0(T) & \coloneqq \{ q_T :~ q_T(\bx) = \hat{q}_{\hT}(\hat{x}) \text{ with }  \hat{q}_{\hT} \in \hQ_0(\hT) \text{ and }x = F_T(\hat{x})\}.
\end{aligned}    
\end{equation}
It follows from the definition of $\hbPhi(\hT)$ in \eqref{def:hbPhi} that
\begin{equation}\label{eq::hatDiv1}
\Hat{\Div}\hbPhi(\hT)=\operatorname{span}\big\{3(1-4\hat{x}_1-2\hat{x}_2),~ 3(1-2\hat{x}_1-4\hat{x}_2)\big\}=\hQ_0(\hT).    
\end{equation}
This shows that $\|\Hat{\Div}\hbphi_{\hT}\|_{L^2(\hT)}$ defines a norm for any $\hbphi_{\hT} \in \hbPhi(\hT)$. By equivalence of norms, one can obtain
\begin{equation}\label{eq::hatDiv2}
\|\hbphi_{\hT}\|_{H^1(\hT)}\leq C\|\Hat{\Div}\hbphi_{\hT}\|_{L^2(\hT)} \text{ for all } \hbphi_{\hT}\in\hbPhi(\hT).
\end{equation}
A combination of \eqref{eq::hatDiv1} and \eqref{eq::hatDiv2} reveals the relationship between the spaces $\bPhi(T)$ and $Q_0(T)$ as follows.
\begin{lemma}
\label{BubbleProperty} 
For any $q_T \in Q_0(T)$, there exists a $\bphi_T\in \bPhi(T)$ such that 
$$\Div \bphi_T =\frac{1}{\det(DF_T \circ F_T^{-1} )}q_T \text{ and } \| \nabla \bphi_T \|_{L^2(T)} \leq C \big\| \frac{1}{\det(DF_T \circ F_T^{-1} )}q_T \big\|_{L^2(T)}.$$
\end{lemma}
\begin{proof}
    Given $q_T\in Q_0(T)$, it follows from the definition of $Q_0(T)$ that there exists a $\hq_{\hT}\in\hQ_0(\hT)$ such that $q_T(x)=\hq_{\hT}(\hat{x})$ with $x=F_T(\hat{x})$.
    According to~\eqref{eq::hatDiv1} and~\eqref{eq::hatDiv2}, there exists a $\hbphi_{\hT} \in \hbPhi(\hT)$, such that
    \begin{equation}
    \label{eq::hatDiv3}
    \Hat{\Div}\hbphi_{\hT} =\hq_{\hT} ~~\text{and}~~\|\hbphi_{\hT}\|_{H^1(\hT)}\leq C\|\hq_{\hT}\|_{L^2(\hT)}.
    \end{equation}
    Let $\bphi_T(x)=A_T(\hat{x})\hbphi_{\hT}(\hat{x})$ with $x = F_T(\hx)$, then $\bphi_T\in\bPhi(T)$. \eqref{pro-AT:div} and \eqref{eq::hatDiv3} give
    \begin{equation}
    \label{eq::hatDiv4}
    \Div\bphi_T(x)=\frac{1}{\det(DF_T(\hat{x}))}\Hat{\Div}\hbphi_{\hT}(\hat{x})=\frac{1}{\det(DF_T \circ F_T^{-1}(x))}q_T(x).    
    \end{equation}
    Lemma~\ref{lem::normH1XThXT}, \eqref{eq::hatDiv3}, Lemma~\ref{lem::NormInThT}, and \eqref{Prop-FT:eq2} yield
    $$
    \begin{aligned}
     \|\nabla\bphi_T\|_{L^2(T)}&\leq Ch_T^{-1}\|\hbphi_{\hT}\|_{H^1(\hT)}\leq Ch_T^{-1}\|\hq_{\hT}\|_{L^2(\hT)}\\
     &\leq Ch_T^{-2}\|q_T\|_{L^2(T)}\leq C \big\| \frac{1}{\det(DF_T \circ F_T^{-1} )}q_T \big\|_{L^2(T)},   
    \end{aligned}
    $$
    which together with \eqref{eq::hatDiv4} completes the proof.
\end{proof}

Following similar procedures as in~\cite[Lemma 4.3]{neilan2021divergence}, one can prove that the piecewise constant
component of $Q_h$ can be controlled by $\bWh$, as shown in the following lemma. For the reader's convenience, the complete proof is provided.

\begin{lemma}\label{controlConst}
    For any $q_h \in Q_h$, there exists a $\bw_h \in\bWh$ such that
    $$
    \int_T   \Div \bw_h -\frac{2q_h |\tT|}{\det(DF_T\circ F_T^{-1})} \dx =0 \text{ and }\|\nabla_h \bw_h \|_{L^2(\Omega_h)}\leq C\|q_h\|_{L^2(\Omega_h)}.
    $$
\end{lemma}
\begin{proof}
Given $q_h \in Q_h$, the definition of $Q_h$ in \eqref{defQinterpolate} shows that there exists a $\tq_h \in \tQ_h$ such that $q_h = \Upsilon \tq_h$. 
Let $\bar{q}_h$ be a piecewise constant function defined on $\calTh$ as
$$
\bar{q}_h|_T=\int_T\frac{2q_h}{\det(DF_T\circ F_T^{-1})} \dx \quad \text{ for all } T \in \calTh,
$$
and $\tilde{\bar{q}}_h$ be a piecewise constant function on $\tcalTh$ such that $\tilde{\bar{q}}_h|_{\tT}=\bar{q}_h|_T$ for all $\tT \in \tcalTh$.
This and \eqref{re-eq:qh} show $\tilde{\bar{q}}_h\in\tQ_h$ since
$$
\int_{\tOmega_h}\tilde{\bar{q}}_h \rd \tbx =\sum_{\tT\in\tcalTh}\tilde{\bar{q}}_h|_{\tT} |\tT|=\sum_{T\in\calTh}\bar{q}_h|_{T} |\tT|
=\sum_{T\in\calTh} \int_T\frac{2q_h|\tT|}{\det(DF_T\circ F_T^{-1})}\dx = 0.
$$
By~\cite[Theorem 4.4]{bernardi2016contiuity} and the properties of $G$, there exists a $\tbw \in\bH_0^1(\tOmega_h)$ such that
\begin{equation}
\label{is2:eq2}
\tDiv\tbw=\tilde{\bar{q}}_h \text{ and  } \|\tnabla\tbw\|_{L^2(\tOmega_h)}\leq C\|\tilde{\bar{q}}_h \|_{L^2(\tOmega_h)}.
\end{equation}
Additionally, for $\tbw\in\bH_0^1(\tOmega_h)$, from the stability proof of the piecewise quadratic-constant pair~\cite{bernardi1985analysis,boffi2013mixed}, there exists a $\tbw_h \in\tbW_h$ such that
\begin{equation}
\label{is2:eq3}
\int_{\tilde{e}}\tbw_h \rd \tilde{s} =\int_{\tilde{e}}\tbw \rd \tilde{s} \, \text{ for all } \tilde{e} \in \tilde{\mathcal{E}}_h,\quad \|\tnabla \tbw_h \|_{L^2(\tOmega_h)}\leq C\|\tnabla\tbw\|_{L^2(\tOmega_h)}.    
\end{equation}
Recall the definition of $\Psi_h^W$ and let $\bw_h =\Psi_h^W\tbw_h \in \bWh$. 

The definition of $\bar{q}_h$, the divergence theorem and the definition of $\tilde{\bar{q}}_h$, and Theorem \ref{Psitbvandbv} (b) and \eqref{is2:eq2} show that 
\begin{equation}
\label{is2:eq4}
\begin{aligned}
& \int_T  \Div\bw_h -\frac{2q_h |\tT|}{\det(DF_T\circ F_T^{-1})}  \dx  =  \int_T\Div\bw_h \dx - \bar{q}_h|_T |\tT|\\
& \qquad  = \int_{\partial T} \bw_h \cdot \bn_T \rd s - \int_{\tT}\tilde{\bar{q}}_h \rd \tbx   = \int_{\partial\tT} \tbw_h \cdot\tilde{\bn}_{\tT} \rd \tilde{s} - \int_{\tT} \tDiv \tbw \rd \tbx = 0,
\end{aligned}    
\end{equation}
where the last equality is obtained by the divergence theorem for $\tDiv\tbw$ and \eqref{is2:eq3}. 
Additionally, the definitions of $\bar{q}_h$ and $\tilde{\bar{q}}_h$, the Cauchy-Schwarz inequality, and the estimate of $\det(DF_T)$ in \eqref{Prop-FT:eq2} yield 
$$
\begin{aligned}
 \|\tilde{\bar{q}}_h \|_{L^2(\tOmega_h)}^2
&=\sum_{\tT\in\tcalTh}\|\tilde{\bar{q}}_h\|_{\tT}^2=\sum_{\tT\in\tcalTh} (\tilde{\bar{q}}_h |_{\tT})^2 |\tT| = \sum_{T\in\calTh} (\bar{q}_h|_T)^2 |\tT|\\
& \leq \sum_{T\in\calTh}  |\tT|  \int_T\frac{4}{\det(DF_T\circ F_T^{-1})^2}\dx \int_T q_h^2 \dx   \leq C \|q
_h\|_{L^2(\Omega_h)}^2. 
\end{aligned}
$$
Theorem \ref{Psitbvandbv} (c), \eqref{is2:eq3}, \eqref{is2:eq2} and this show that   
$$
\| \nabla_h \bw_h \|_{L^2(\Omega_h)} \leq C\|\tnabla\tbw_h \|_{L^2(\tOmega_h)} \leq C\|\tnabla\tbw\|_{L^2(\tOmega_h)} \leq C\|\tilde{\bar{q}}_h\|_{L^2(\tOmega_h)} \leq C \|q_h\|_{L^2(\Omega_h)}^2.     
$$
This, combined with \eqref{is2:eq4}, completes the proof.
\end{proof}

The global inf-sup condition follows directly from the above two lemmas.

\begin{theorem}[Inf-sup condition]
\label{thm:inf-sup}
    It holds that
    $$
    \sup_{\bv_h  \in \bVh \setminus\{0\}}\frac{\int_{\Omega_h}q_h \Div_h \bv_h \dx} {\|\nabla_h \bv_h \|_{L^2(\Omega_h)}}\geq C\|q_h\|_{L^2(\Omega_h)} \quad \text{ for all } q_h \in Q_h.
    $$
\end{theorem}
\begin{proof}
    Given $q_h \in Q_h$, Lemma \ref{controlConst} shows that there exists $\bw_h \in \bWh$ such that
    \begin{equation}
    \label{is3:eq1}
    \int_T \Div\bw_h -\frac{2q_h|\tT|}{\det(DF_T\circ F_T^{-1})}  \dx =0 \text{ and }\|\nabla_h \bw_h\|_{L^2(\Omega_h)}\leq C\|q_h\|_{L^2(\Omega_h)}.
    \end{equation}
    For any $T \in \calTh$, let $\bw_T = \bw_h|_T \in \bW(T)$ and $q_T = q_h|_T \in Q(T)$. Claim that $q^0_T \coloneqq \det(DF_T\circ F_T^{-1}) \Div \bw_T - 2 q_T |\tT| \in Q_0(T)$, which will be shown later. Lemma~\ref{BubbleProperty} shows that there exists unique $\bphi_T \in \bPhi(T)$ such that
    \begin{equation}
    \label{is3:eq2}
    \Div\bphi_T = \frac{1}{\det(DF_T\circ F_T^{-1})} q^0_{T} = \Div\bw_T -\frac{2q_T|\tT|}{\det(DF_T\circ F_T^{-1})},    
    \end{equation}
    and
    \begin{equation}
    \label{is3:eq3}
    \|\nabla \bphi_T \|_{L^2(T)} \leq C ( \| \Div \bw_T \|_{L^2(T)} + \|q_T\|_{L^2(T)} ).        
    \end{equation}
    Define $\bphi_h|_T = \bphi_T$ for all $T \in \calTh$ and let $\bv_h = \bw_h - \bphi_h$. By \eqref{is3:eq2}, it holds that
    \begin{equation}
    \label{is3:eq5}
    \begin{aligned}
    \int_{\Omega_h}q_h \Div_h \bv_h \dx &= \sum_{T\in \calTh} \int_T \Div \bv_T \, q_T \dx = \sum_{T \in \calTh} \int_T \frac{2q_T^2|\tT|}{\det(DF_T\circ F_T^{-1})} \dx   \\
    & \geq C \sum_{T \in \calTh} \int_T q_T^2 \dx = C \| q \|_{L^2(\Omega_h)}^2.
    \end{aligned}    
    \end{equation}
    The combination of \eqref{is3:eq1} and \eqref{is3:eq3} gives $\| \nabla_h \bv_h\|_{L^2(\Omega_h)} \leq C \| q_h\|_{L^2(\Omega_h)}$. This, together with \eqref{is3:eq5}, proves the inf-sup condition. 
    
    It remains to show that $q_T^0 \in Q_0(T)$.
    By the definitions of $\bW(T)$ and $Q(T)$ that there exist $\hbw_{\hT}$ and $\hq_{\hT}$ such that $\bw_T(\bx) = A_T(\hat{x})\hbw_{\hT}(\hat{x})$ and $q_T(\bx)=\hq_{\hT}(\hat{x}) \text{ with }x = F_T(\hat{x})$.
    Define $\hq^0_{\hT} \coloneqq \hDiv \hbw_{\hT} - 2 \hq_{\hT}| \tT| \in \hQ(\hT)$. It follows from \eqref{pro-AT:div} that $q_T^0(\bx)= \hq_{\hT}^0(\hat{x})$ with $\bx = F_T(\hat{x})$, which implies $q_T^0 \in Q(T)$. This, the definition of $q_T^0$ and \eqref{is3:eq1} lead to
    $$
    \int_{\hT} \hq_{\hT}^0 \rd \hat{x}  = \int_{T} \frac{q_T^0}{\det(DF_T\circ F_T^{-1})} \dx =  \int_T \left( \Div\bw_T -\frac{2q_T|\tT|}{\det(DF_T\circ F_T^{-1})} \right) \dx =0,
    $$
    which implies $\hq_{\hT}^0 \in \hQ_0(\hT)$. Thus, $q_{T}^0 \in Q_0(T)$.
\end{proof}
\begin{remark}
The idea for ensuring the inf-sup condition in this paper is analogous to that in~\cite{neilan2021divergence}.
Specifically, Clough--Tocher refinements are adopted in \cite{neilan2021divergence} to ensure that there are sufficient functions on each element to control the component of $Q_h$ orthogonal to piecewise constants, while nonconforming bubble functions are adapted in this paper for the same purpose.
\end{remark}

\begin{theorem}
There exists a unique solution $(\bu_h, p_h) \in \bVh \times Q_h$ satisfying~\eqref{eqn:FEM}.
\end{theorem}
\begin{proof}
Note that Lemma \ref{col::H1Norm} shows that $\| \nabla_h \bv_h \|_{L^2(\Omega_h)}$ is a norm for all $\bv_h \in \bVh$.
By the Babusk\v{a}--Brezzi theory~\cite{boffi2013mixed} and Theorem \ref{thm:inf-sup}, the proof is completed.
\end{proof}

Although $\Div_h\bVh$ is not a subset of $Q_h$, the lemma below shows that the finite element method \eqref{eqn:FEM} yields element-wise divergence-free velocity approximations. The proof follows a similar procedure to that in~\cite[Lemma 5.2]{neilan2021divergence}. For completeness, the detailed proof is provided.
\begin{lemma}
\label{lem:div-free}
Suppose $\bu_h \in \bVh$ satisfying \eqref{eqn:FEM2}. Then $\Div_h \bu_h = 0$ in $\Omega_h$.    
\end{lemma}
\begin{proof}
By Theorem \ref{the::VhEqXhDirSumPhih}, there exist unique $\bw_h \in \bWh$ and $\bphi_h \in \bPhi_h$ such that $\bu_h = \bw_h + \bphi_h$. Denote $\bw_T = \bw_h |_T \in \bW(T)$ and $\bphi_T = \bphi_h|_T \in \bPhi(T)$ for any $T \in \calTh$. It follows from the definitions of $\bW(T)$ and $\bPhi(T)$ that there exist $\hbw_{\hT}$ and $\hbphi_{\hT}$ such that $\bw_T(x)= A_T(\hx)\hbw_{\hT}(\hx)$ and $\bphi_T(x)=A_T(\hx)\hbphi_{\hT}(\hx)$ with $x = F_T(\hx)$.
Define $q_h$ to be the piecewise function as
$$
q_h|_T(x)=\frac{1}{2|\tT|}(\Hat{\Div}\hbw_T + \Hat{\Div}\hbphi_T)(\hat{x}) \text{ with }
x=F_T(\hat{x}) \text{ and } T=G_h(\tT).
$$
Since $\bw_h \in \bH_0(\Div;\Omega_h)$ in \eqref{pro-Wh:space}, the divergence theorem gives
\begin{equation*}
\sum_{T \in \calTh} \int_{T}   \Div \bw_T \dx =  \int_{\Omega_h} \Div \bw_h \dx = \int_{\partial \Omega_h} \bw_h \cdot \bn \ds  = 0.
\end{equation*}
Additionally, on each element $T \in \calTh$, the divergence theorem and Lemma \ref{lem:pro-AT} show
\begin{equation*} 
\int_{T} \Div \bphi_T \dx =\int_{\partial T} \bphi_T \cdot \bn_T \ds =\int_{\partial \hT} \hbphi_{\hT} \cdot \hbn_{\hT} \hds   = 0,
\end{equation*}
where the last equality holds due to $\hbphi_{\hT} \in \hbPhi(\hT)$ vanishes on the Gauss--Legendre points $\mathcal{G}_{\hT}$. 
The above two equalities combine with the property of the Piola transform \eqref{pro-AT:div} show
$$
\sum_{T\in\calTh}\int_T \frac{2|\tT|q_h}{\det(DF_T\circ F_T^{-1})}\dx=\sum_{T\in\calTh}\int_T \Div \bw_T + \Div\bphi_T \dx  = 0,
$$
which implies $q_h\in Q_h$ by Remark~\ref{re:qh}.
This, \eqref{eqn:FEM2}, and \eqref{pro-AT:div} show
$$
\begin{aligned}
0&=\int_{\Omega_h}q_h\Div_h\bu_h\dx = \sum_{T\in \calTh} \int_T q_h \Div(\bw_T+\bphi_T)  \dx \\
& =  \sum_{T\in \calTh}\frac{1}{2|\tT|}  \int_{\hT} \Hat{\Div}(\hbw_T +\hbphi_T)   \frac{\Hat{\Div}(\hbw_T +\hbphi_T)}{\det(DF_T)}  \det(DF_T) \hdx
\end{aligned}
$$
Thus, $\Hat{\Div}(\hbw_T +\hbphi_T)=0$ for all $T \in \calTh$, which implies $\Div_h \bu_h=0$ in $\Omega_h$.
\end{proof}

\section{Error estimate}
\label{sec::error_estimate}

This section details the convergence analysis for the discrete problem \eqref{eqn:FEM}. The consistency error, arising from the nonconforming nature of the curved elements, is bounded by exploiting the weak continuity of the velocity space $\bVh$.

\subsection{Weak Continuity Properties of $\bVh$}\label{sec::weak_continuity}
This subsection utilizes the connection between $\bVh$ and $\bVISO$ given in Lemma \ref{lem::EhError} to estimate the consistency error incurred by the approximation of $\bH^1_0(\Omega_h)$ with $\bVh$.

Given $e\in\calEh$, let $\omega_e$ denote the set of elements $T\in\calTh$ such that $e\subset\partial T$.
For any $e \in \calEh$, take $T\in\omega_e$. Let $\hat{e} \subset \partial \hT$ such that $F_T(\hat{e})=e$ and define $\Je(\hat{x})=\|A_T(\hat{x})^{-\intercal}\hat \bn_{\hat{e}}\|_{\mathbb{R}^2}$ for $\hat{x}$ on $\hat{e}$.
For any $v\in L^2(e)$, by~\cite[(2.1.62)]{boffi2013mixed}, it holds that 
\begin{equation}
\label{pro-Je}
\int_e v \ds = \int_{\hat{e}} \hat{v} J_e \hds \quad \text{ with } \hat{v}(\hat{x})=v(\bx) \text{ and } x=F_T(\hat{x}).
\end{equation}
For any $v\in L^2(e)$, define the interpolation operator $\Pe$ as
\begin{equation}
\label{DefinePe}
\Pe v(\bx) = \frac{1}{\Je(\hat{x})}\Phe(\Je(\hat{x}) \hat{v}(\hat{x})) \text{ with } x=F_T(\hat{x}) \text{ and } \hat{v}(\hat{x})=v(\bx),    
\end{equation}
where the $L^2$-projection $\Phe:L^2(\hat{e})\to P_1(\hat{e})$ is defined by
\begin{equation}
\label{defPhe}
\int_{\hat{e}}(\hat{v}-\Phe\hat{v}) \hat{p}\hds=0 \, \text{ for all } \hat{p}\in P_1(\hat{e}).   
\end{equation} 
Note that if $e$ is an interior edge, then $e$ is straight and $\Je(\hat{x})$ is constant, see~\cite[Lemma 2.4]{neilan2021divergence}. In this case, $\Pe$ coincides with the standard $L^2$-projection from $L^2(e)$ onto $P_1(e)$. If $e$ is a boundary edge, there exists a unique $T \in \omega_e$, which implies that $\Pe$ is well-defined. Using arguments analogous to those in~\cite[Lemma 3]{coruzeix1873conforming}, one can obtain the following estimate of $\Pe$; its proof is provided in Appendix \ref{app::Prooflem::EstimateNoncon}.
\begin{lemma}\label{lem::EstimateNoncon}
There exists a constant $C>0$, independent of $T$, such that
\begin{equation}\label{ErrorPe}
\left|\int_e\varphi(v-\Pe v)\ds\right|\leq Ch_T^{m+1}|\varphi|_{H^1(T)}\|v\|_{H^{m+1}(T)},
\end{equation}
holds for all $\varphi\in H^1(T)$ and $v\in H^{m+1}(T)$ with $m=0,1$.
\end{lemma}

Define the jump of a vector-valued function $\bv$ across $e$ as
$$
[\bv]|_e=\sum_{T\in \omega_e}(\bv_T\otimes\bn_T)|_e,
$$
where $\bv_T=\bv|_T$ and $\bn_T$ is the outward unit normal of $\partial T$ restricted to $e$, and $(\boldsymbol{a}\otimes \bb)_{i,j}\coloneqq\ba_i b_j$. 

\begin{lemma}
\label{I3Claim}
For any $e\in\calEh$, it holds that
\begin{equation*}
\Big|\int_{e}[\bv_h]\ds\Big|\leq C\max_{T\in \omega_e}h_T^3\sum_{T\in \omega_e}\|\nabla\bv_h\|_{L^2(T)}.
\end{equation*}
\end{lemma}

\begin{proof}
By Theorem \ref{the::VhEqXhDirSumPhih}, there exist unique $\bw_h\in\bWh$ and $\bphi_h\in\bPhi_h$ such that $\bv_h=\bw_h+\bphi_h$. By construction, $[\bw_h]|_e$ vanishes at the endpoints and the midpoint of $e$, and $[\bphi_h]|_e$ vanishes at the two Gauss--Legendre points of $e$. This, combined with Simpson's and Gauss--Legendre's quadrature rule, shows
\begin{equation}
\label{vBIIjumpInit}
\begin{aligned}
\Big|\int_e[\bv_h] \rd s \Big|
&\leq C|e|^5\big(|[\bw_h]|_{W^{4,\infty}(e)}+|[\bphi_h]|_{W^{4,\infty}(e)}\big)\\
&\leq C\sum_{T\in \omega_e}h_T^5(|\bw_h|_{W^{4,\infty}(T)}+|\bphi_h|_{W^{4,\infty}(T)}).
\end{aligned}    
\end{equation}
For any $T \in \omega_e$, let $\bw_T = \bw_h|_T$. There exists $\hbw_{\hT} \in \hbW(\hT)$ such that $\bw_T(x)= A_T(\hx)\hbw_{\hT}(\hx)$ with $x = F_T(\hx)$ by definition. Note that $\hbw_{\hT}$ is a quadratic polynomial. Lemma \ref{lem::normH1XThXT} and equivalence of norms in $\hT$ show
\begin{equation*}
|\bw_T|_{W^{4,\infty}(T)}\leq C h_T^{-5}\sum_{j=0}^2 h_T^{4-j} |\hbw_{\hT}|_{W^{j,\infty}(\hT)}   \leq C h_T^{-3} \|\hbw_{\hT}\|_{L^2(\hT)} \leq C h_T^{-3}\|\bw_T\|_{L^2(T)}.
\end{equation*}
Similar arguments for $\bphi_h$ yield $|\bphi_h|_{W^{4,\infty}(T)}\leq C h_T^{-3}\|\bphi_h\|_{L^2(T)}$. Substituting these two estimates into \eqref{vBIIjumpInit} shows
\begin{equation}\label{vBIIjump}
\Big|\int_e[\bv_h] \rd s \Big|\leq C\sum_{T\in \omega_e}h_T^2(\|\bw_h\|_{L^2(T)}+\|\bphi_h\|_{L^2(T)}).
\end{equation}

\textbf{Case 1:} 
If $e$ is a boundary edge, $\bw_h|_e$ vanishes on both the two endpoints and the midpoint of $e$. By Lemma~\ref{lem:VTestimate1}, it holds that
\begin{equation}
\label{vBIIjump1}
\Big|\int_e[\bv_h] \rd s \Big|\leq C\sum_{T\in \omega_e}h_T^2\|\bv_h\|_{L^2(T)} \leq C \sum_{T\in \omega_e}h_T^3\|\nabla \bv_h\|_{L^2(T)} .
\end{equation}

\textbf{Case 2:}
Let $e$ is an interior edge with $e = \partial T_+ \cap \partial T_-$, where $\omega_e = \{T_+,T_-\}$. If $T_+$ and $T_-$ are affine, then $[\bw_h]|_e$ and $[\bphi_h]|_e$ are quadratic polynomials. This and \eqref{vBIIjumpInit} show $\int_{e}[\bv_h]\ds=0$, implying \eqref{I3Claim}. In the following, suppose that $T_+$ has a curved edge $e^\prime \subset \partial\Omega_h$. Since $\bw_h|_{e^\prime}=0$, it follows $\bw_h(a)=0$ with $a = e \cap e^\prime$. Using \eqref{bvcbvbL2Bd} from Lemma~\ref{lem:VTestimate1}, the first inequality in \eqref{vBIIjump1} still holds. It remains to show $\sum_{T\in \omega_e}\|\bv_h\|_{L^2(T)}\leq C\sum_{T\in \omega_e}h_T\|\nabla\bv_h\|_{L^2(T)}$. Let $\bs_h = \Eh\bv_h$. Lemma~\ref{lem::EhError} gives
\begin{equation}
\label{eq::vh2Eh}
\begin{aligned}
\sum_{T\in \omega_e}\|\bv_h\|_{L^2(T)}&\leq \sum_{T\in \omega_e}(\|\bv_h-\bs_h\|_{L^2(T)}+\|\bs_h\|_{L^2(T)})\\
&\leq \sum_{T\in \omega_e}(Ch_T^2\|\nabla\bv_h\|_{L^2(T)}+\|\bs_h\|_{L^2(T)}).  
\end{aligned}    
\end{equation}
Let $\hT_+$ with vertices $(0,0)$, $(1,0)$ and $(0,1)$ and $\hT_-$ with vertices $(0,0)$, $(0,1)$ and $(-1,0)$ be the reference triangles. There exists an isoparametric map $F_{T_\pm}:~\hT_\pm\to T_\pm$ satisfying $F_{T_+}^{-1}(e)=F_{T_-}^{-1}(e)$. Define the piecewise quadratic polynomial $\hat{\bs}_h$ on $\hT_+\cup\hT_-$ satisfying $\hat{\bs}_h|_{\hT_\pm}(\hat{x})=\bs_h|_{T_\pm}(\bx)$ with $x=F_{T_\pm}(\hat{x})$ and $\hat{x}\in \hT_\pm$. Note that $\bs_h \in \bVh^{iso}$, the mapping $\hat{\bs}_h\to \|\hat{\nabla}\hat{\bs}_h\|_{L^2(\hT_+)}+\|\hat{\nabla}\hat{\bs}_h\|_{L^2(\hT_-)}$ defines a norm. Using Lemma~\ref{lem::NormInThT} and the equivalence of norms yields
$$
\begin{aligned}
&\sum_{T\in \omega_e}\|\bs_h\|_{L^2(T)}
\leq C(h_{T_+}\|\hat{\bs}_h\|_{L^2(\hT_+)}+h_{T_-}\|\hat{\bs}_h\|_{L^2(\hT_-)})\\
& \quad \quad \leq 
C(h_{T_+}\|\hat{\nabla}\hat{\bs}_h\|_{L^2(\hT_+)}+h_{T_-}\|\hat{\nabla}\hat{\bs}_h\|_{L^2(\hT_-)})\leq C\sum_{T\in\omega_e}h_T\|\nabla\bs_h\|_{L^2(T)}.
\end{aligned}
$$
Substituting this into \eqref{eq::vh2Eh} completes the proof. 
\end{proof}
The lemma below shows the consistency error induced by the space $\bVh$. Its proof relies on the approximation property of $\bE_h$, the weak continuity property of $\bVh^{iso}$, and Lemma~\ref{I3Claim}. 
\begin{lemma}
\label{lem::nonconErrors}
Define $\bsig = \nu \nabla \bu - p \mathrm{I}$ on $\mathbb{R}^{2}$ by \eqref{extenUandP}. For all $\bv_h\in\bVh$, it holds that
\begin{equation*}
\Big|\sum_{e\in\calEh}\int_e\bsig:[\bv_h]\rd s\Big| \leq Ch^2 (\nu \|\bu\|_{H^3(\Omega)} + \|p\|_{H^2(\Omega)})\|\nabla_h \bv_h\|_{L^2(\Omega_h)}.
\end{equation*}
\end{lemma}
\begin{proof}
Let $\bG_e\in\mathbb{R}^{2\times2}$ be the average of $\bsig$ on $e\in\calEh$ and $\Eh\bv_h \in \bVISO$ be the interpolant of $\bv_h$ defined in \eqref{def:Eh}. A triangle inequality leads to
\begin{equation}\label{lemErr0}
\begin{aligned}
& \Big|\sum_{e\in\calEh}\int_e\bsig:[\bv_h]\rd s\Big| \leq \Big|\sum_{e\in\calEh}\int_e(\bsig-\bG_e):[\bv_h-\Eh\bv_h]\rd s\Big|\\
&\quad+\Big|\sum_{e\in\calEh}\int_e(\bsig-\bG_e):[\Eh\bv_h]\rd s\Big|+\Big|\sum_{e\in\calEh}\int_e\bG_e:[\bv_h]\ds\Big|
=: I_{1}+I_{2}+I_{3}.
\end{aligned}
\end{equation}
\textbf{Estimate for $I_1$}.
For any $T\in \omega_e$, a standard interpolation estimate and the trace theorem yield
\begin{equation*}
\label{sigGeError}
h_e^{-1}\|\bsig-\bG_e\|^2_{L^2(e)}\leq C|\bsig|_{H^1(T)} \text{ with } h_e=\operatorname{diam}(e).
\end{equation*}
This, the Cauchy--Schwarz inequality, the trace theorem, and Lemma \ref{lem::EhError} give
\begin{equation*}
\begin{aligned}
 I_{1}&\leq \Big(\sum_{e\in\calEh}h_e^{-1}\|\bsig-\bG_e)\|^2_{L^2(e)}\Big)^{1/2}\Big(\sum_{e\in\calEh}h_e\|[\bv_h-\Eh\bv_h]\|^2_{L^2(e)}\Big)^{1/2}\\
&\leq Ch^2|\bsig|_{H^1(\Omega_h)}\|\nabla\bv_h\|_{L^2(\Omega_h)}.   
\end{aligned}
\end{equation*}
\textbf{Estimate for $I_2$}.
Note that $\Eh\bv_h \in \bVh^{iso}$ implies that the jump $[\Eh\bv_h]|_e$ vanishes at the two Gauss--Legendre points for all $e \in \calEh$. Combining this with \eqref{pro-Je}--\eqref{defPhe}, one can insert the $L^2$-projection $\Pe$ and apply Lemma \ref{lem::EstimateNoncon} with $m=1$ to obtain the following estimate
\begin{equation*}
\begin{aligned}
I_{2}&= \Big| \sum_{e\in\calEh}\int_e(\bsig-\bG_e-\Pe(\bsig-\bG_e)):[\Eh\bv_h]\rd s\Big|\\
&\leq \sum_{e\in\calEh}\sum_{T\in \omega_e}Ch_T^{2}\|\bsig-\bG_e\|_{H^2(T)}|\Eh\bv_h|_{H^1(T)}\\
&\leq \sum_{e\in\calEh}\sum_{T\in \omega_e}Ch_T^{2}(\|\bsig\|_{H^2(T)}+\|\bG_e\|_{H^2(T)})(|\Eh\bv_h-\bv_h|_{H^1(T)}+|\bv_h|_{H^1(T)}).
\end{aligned}
\end{equation*}
Since $\bG_e$ is a constant on $T$ and $|\bG_e|\leq C\|\bsig\|_{L^\infty(\Omega_h)}\leq C\|\bsig\|_{H^2(\Omega_h)}$, it follows $\|\bG_e\|_{H^2(T)} =|\bG_e|\, |T|^{1/2} \leq Ch_T\|\bsig\|_{H^2(\Omega_h)}$.
This, the Cauchy–Schwarz inequality, and Lemma~\ref{lem::EhError} lead to
\begin{equation*}
\begin{aligned}
I_{2}
&\leq Ch^2\|\bsig\|_{H^2(\Omega_h)}\|\nabla\bv_h\|_{L^2(\Omega_h)}.
\end{aligned}
\end{equation*}
\textbf{Estimate for $I_3$}. 
Lemma \ref{I3Claim} and the Cauchy–Schwarz inequality lead to
\begin{equation*}
\label{lemErr3}
\begin{aligned}
I_{3}
&\leq\sum_{e\in\calEh}\Big|\bG_e:\int_e[\bv]\ds\Big|\\
&\leq C\|\bsig\|_{L^\infty(\Omega_h)}\Big(\sum_{e\in\calEh}h_e\Big)^{1/2}\Big(\sum_{e\in\calEh}h_e^{-1}\Big|\int_e[\bv]\ds\Big|^2\Big)^{1/2}\\
&\leq Ch^{5/2}\|\bsig\|_{H^2(\Omega_h)}\|\nabla\bv\|_{L^2(\Omega_h)}.
\end{aligned}
\end{equation*}   

Finally, combining the estimates for $I_1$, $I_2$, and $I_3$ with the decomposition \eqref{lemErr0} completes the proof.
\end{proof}

\subsection{Convergence Analysis}
Define the divergence-free subspace of $\bVh$ as
$$
\bZ_h \coloneqq \{ \bv_h \in \bVh:~ \Div_h \bv_h = 0  \text{ on } \Omega_h \}. 
$$
By Lemma \ref{lem:div-free} and \eqref{eqn:FEM1}, the discrete velocity solution $\bu_h \in \bVh$ can be uniquely determined by the following problem: Find $\bu_h \in \bZ_h$ such that 
$$
a_h(\bu_h, \bv_h) \coloneqq \int_{\Omega_h} \nu \nabla_h \bu_h :~\nabla_h \bv_h \dx = \int_{\Omega_h} \bft_h \cdot \bv_h \dx \text{ for all } \bv_h \in \bZ_h.
$$
Define the  dual norm on $\bZ_h$ as
$$
| \bft - \bft_h |_{\bZ_h^*} \coloneqq \sup_{\bv_h \in \bZ_h \setminus \{0\} } \frac{\int_{\Omega_h} (\bft - \bft_h) \cdot \bv_h \dx}{\|\nabla_h \bv_h \|_{L^2(\Omega_h)} }.
$$
The following theorem provides error estimates of the solution $(\bu_h, p_h) \in \bVh \times Q_h$ for solving the Stokes equation in \eqref{eqn:FEM}. 
\begin{theorem}
\label{thm:es1}
Let $(\bu,p)\in H^3(\Omega)\times H^{2}(\Omega)$ satisfy \eqref{var: continuous}. Then it holds that
\begin{align}
\label{es1:velocity}
\| \nabla \bu- \nabla_h \bu_h \|_{L^2(\Omega_h)} &\leq C h^2 (\| \bu \|_{H^3(\Omega)} + \nu^{-1}\|p\|_{H^2(\Omega)}) + C \nu^{-1} | \bft- \bft_h|_{\bZ_h^*},
\end{align}   
and
\begin{equation}
\label{es1:pressure}
\begin{aligned}
\|p-p_h\|_{L^2(\Omega_h)}\leq &C\big(\nu\| \nabla \bu - \nabla_h \bu_h \|_{L^2(\Omega_h)}+ \nu h^2\|\bu\|_{H^3(\Omega)}
\\&+h^2\|p\|_{H^2(\Omega)}+\inf_{q_h \in Q_h}\|p-q_h\|_{L^2(\Omega)}+\|\bft-\bft_h\|_{L^2(\Omega_h)}
\big).
\end{aligned}
\end{equation}
\end{theorem}
\begin{proof}
From the standard theory of nonconforming and mixed finite element methods (e.g.,~\cite{brenner2008the}), the inf-sup condition of Theorem \ref{thm:inf-sup} together with the approximation property of Lemma \ref{lem:uniqueXT} gives,
\begin{equation}
\label{es1:eq1}
\begin{aligned}
& \nu \| \nabla \bu- \nabla_h \bu_h\|_{L^2(\Omega_h)} \\
&\leq  \inf_{\bv_h \in \bZ_h} \nu \| \nabla \bu - \nabla_h \bv_h \|_{L^2(\Omega_h)}+\sup_{\bv_h \in \bZ_h\setminus \{0\}} \frac{ a_h(\bu_h -\bu, \bv_h)}{ \| \nabla_h\bv_h\|_{L^2(\Omega_h)}} \\
& \leq C \inf_{\bv_h \in \bVh} \nu \| \nabla \bu - \nabla_h \bv_h \|_{L^2(\Omega_h)}+\sup_{\bv_h \in \bZ_h\setminus \{0\}} \frac{ a_h(\bu_h -\bu, \bv_h)}{ \| \nabla_h\bv_h\|_{L^2(\Omega_h)}} \\
& \leq Ch^2\nu\|\bu\|_{H^3(\Omega)}+\sup_{\bv_h \in \bZ_h\setminus \{0\}} \frac{ a_h(\bu_h -\bu, \bv_h)}{ \| \nabla_h\bv_h\|_{L^2(\Omega_h)}}.
\end{aligned} 
\end{equation}
Set $\bsig \coloneqq \nu\nabla\bu-pI$ and 
recall $-\nu \Delta \bu + \nabla p = \bft$ in $\mathbb{R}^2$. For any $\bv_h \in \bZ_h$, it holds that
\begin{equation}
\label{erreq1}
a_h(\bu_h -\bu, \bv_h) = -\int_{\Omega_h} \Div\bsig \cdot \bv_h \dx - a_h(\bu, \bv_h) + \int_{\Omega_h} (\bft_h - \bft) \cdot \bv_h \dx.    
\end{equation}
Integrating by parts element‑wise yields
\begin{equation}
\label{erreq2}
\begin{aligned}
\int_{\Omega_h} \Div\bsig \cdot \bv_h \dx= \sum_{T \in \mathcal{T}_h}  \int_{T} \Div\bsig\cdot \bv_h \dx
=\sum_{e \in \calEh} \int_{e} \bsig:[ \bv_h ]\rd s-a_h(\bu,\bv_h).
\end{aligned}    
\end{equation}
Substituting \eqref{erreq2} into \eqref{erreq1} gives
\begin{equation}\label{uError2}
a_h(\bu_h -\bu, \bv_h)=-\sum_{e\in\calEh}\int_e\bsig:[\bv_h]\rd s+\int_{\Omega_h}(\bft_h-\bft)\cdot\bv_h \dx.
\end{equation}
Thus, \eqref{es1:velocity} follows form \eqref{es1:eq1}, \eqref{uError2} and the consistency error in Lemma \ref{lem::nonconErrors}.

For any $q_h \in Q_h$, a triangle inequality and the inf-sup condition in Theorem \ref{thm:inf-sup} lead to
\begin{equation}
\label{ieq:estimateq}
\begin{aligned}
\|p-p_h\|_{L^2(\Omega_h)}&\leq\|p-q_h\|_{L^2(\Omega_h)}+\|p_h-q_h\|_{L^2(\Omega_h)}\\
&\leq \|p-q_h\|_{L^2(\Omega_h)}+\sup_{\bv_h \in\bVh\setminus\{0\}}\frac{\int_{\Omega_h}\Div_h \bv_h \,  (p_h -q_h) \dx}{\|\nabla_h \bv_h \|_{L^2(\Omega_h)}}.
\end{aligned}    
\end{equation}
Using the discrete problem \eqref{eqn:FEM1} and an integration by parts analogous to \eqref{erreq2}, one can obtain, for any $\bv_h\in\bVh$,
$$
\begin{aligned}
& \int_{\Omega_h} \Div_h \bv_h (p_h -q_h) \dx \\
& = a_h(\bu_h, \bv_h)- \int_{\Omega_h} \Div\bv_h q_h \dx-\int_{\Omega_h} (\bft_h-\bft) \cdot \bv_h \dx-\int_{\Omega_h} \bft \cdot \bv_h \dx\\
&=a_h(\bu_h-\bu,\bv_h)+\int_{\Omega_h} \Div_h\bv_h (p-q_h) \dx+\int_{\Omega_h}(\bft-\bft_h)\cdot\bv_h \dx+\sum_{e\in\calEh}\int_e\bsig:[\bv_h]\ds.
\end{aligned}
$$
Applying Lemma \ref{lem::nonconErrors} to the above equality and Lemma \ref{col::H1Norm} show
$$
\begin{aligned}
\int_{\Omega_h} \Div \bv_h &(p_h -q_h) \dx
\leq C\Big(\nu\|\nabla \bu- \nabla_h \bu_h\|_{L^2(\Omega_h)}+\|p-q_h\|_{L^2(\Omega_h)}
\\
&+h^2(\nu\|\bu\|_{H^3(\Omega)}+\|p\|_{H^2(\Omega)})+\|\bft-\bft_h\|_{L^2(\Omega_h)}\Big)\|\nabla_h \bv_h\|_{L^2(\Omega_h)}.   
\end{aligned}    
$$
Substituting this into \eqref{ieq:estimateq} completes the pressure estimate~\eqref{es1:pressure}.
\end{proof}

Assume that $\bft$ is sufficiently smooth. If $\bft_h$ is taken, for example, as its quadratic nodal (isoparametric) interpolant, then Lemma~\ref{col::H1Norm} gives 
$$| \bft- \bft_h|_{\bZ_h^*}\leq C\|\bft-\bft_h\|_{L^2(\Omega_h)}\leq Ch^{3}\|\bft\|_{H^3(\Omega)}.$$
Substituting this into Theorem~\ref{thm:es1} leads to the following velocity error estimate:
\begin{equation}
\label{es1:velocityByIntf}
\begin{aligned}
\| \nabla(u-u_h) \|_{L^2(\Omega_h)}\leq& Ch^2 \big(\| \bu \|_{H^3(\Omega)} + \nu^{-1}\|p\|_{H^2(\Omega)} + \nu^{-1}h\|\bft\|_{H^3(\Omega)}\big),
\end{aligned}    
\end{equation}
and the pressure error estimate:
\begin{equation*}
\begin{aligned}
\|p-p_h\|_{L^2(\Omega_h)}\leq& C\inf_{q\in Q_h}\|p-q\|_{L^2(\Omega_h)}\\
&+ Ch^2\big(\nu\|\bu\|_{H^3(\Omega)}+\|p\|_{H^{2}(\Omega)}+h\|\bft\|_{H^3(\Omega)}\big).
\end{aligned}    
\end{equation*}    
The presence of the terms $\nu^{-1} h^2\big(\|p\|_{H^2(\Omega)}+ h \|\bft\|_{H^3(\Omega)} \big)$ in \eqref{es1:velocityByIntf} indicates that the velocity error remains coupled to the pressure scaling and inversely proportional to the viscosity. Consequently, a direct application of the pair $(\bVh, Q_h)$ fails to provide a pressure-robust scheme, as large pressure gradients or small viscosity values may significantly degrade the accuracy of the velocity approximation.

\section{Pressure-Robust scheme}
\label{sec::pressure_robust_scheme}

This section proposes a modified scheme by replacing the Galerkin source term in \eqref{eqn:FEM1} with $\int_{\Omega_h}\bft_h\cdot\Pi_h\bv_h\dx$, where the computable approximation $\bft_h$ is introduced in~\cite[Section~6]{neilan2021divergence}. This section focuses on the construction of the velocity reconstruction operator $\Pi_h$ on curved domains to ensure pressure-robustness.

\subsection{Computable approximation $\bft_h$}
Define the rot operator $\rot\bv=\frac{\partial v_2}{\partial x_1}-\frac{\partial v_1}{\partial x_2}$ and the associated Hilbert space
$$
\bH(\rot;\Omega_h)\coloneqq\big\{\bv\in\bL^2(\Omega_h):~\rot\bv\in L^2(\Omega_h)\big\}.
$$
The following theorem comes from~\cite[Theorem 6.1]{neilan2021divergence}, which provides a computable
approximation $\bft_h$ for any $\bft\in\bH^3(\Omega)$.
\begin{theorem}
\label{the::PiY}
There exist finite element spaces $\bY_h\subset \bH(\rot;\Omega_h)$, $\Sigma_h\subset H_0^1(\Omega_h)$ with respect to the partition $\calTh$ and operators $\Pi_h^Y:~\bH^2(\Omega)\to\bY_h$ and $\Pi_h^\Sigma:H^3(\Omega)\to\Sigma_h$ such that
\begin{equation*}
\Pi_h^Y\nabla p = \nabla\Pi_h^\Sigma p \quad  \text{ for all } p\in H^3(\Omega).
\end{equation*}
Moreover, for any $\bft\in\bH^3(\Omega)$, the following estimates hold:
\begin{align}
\label{PiYEst}
\|\bft-\Pi_h^Y\bft\|_{L^2(\Omega_h)}&\leq Ch^2\|\bft\|_{H^3(\Omega)}.
\end{align}
Here, $\bft$ in the left-hand sides is $H^3$-extensions of $\bft|_{\Omega}$.
\end{theorem}

\subsection{Velocity reconstruction operator $\Pi_h$}
\label{ParRTSpace}
Based on the framework presented in~\cite{linke2012a,linke2014on,linke2016robust},
this subsection constructs the velocity reconstruction operator $\Pi_h$ on the curved domain $\Omega_h$. This operator~$\Pi_h$ maps the functions in $\bVh$ to the conforming-$\bH(\Div;\Omega_h)$ first-order parametric Raviart–Thomas spaces~\cite{bertrand2014first}:
\begin{equation}
\label{def:Rh}
\begin{aligned}
\bR_h\coloneqq\{\bv_h \in \bH_0(\Div;\Omega_h):~\bv_h|_T\in\bR(T) \text{ for all } T\in\calTh \},    
\end{aligned}
\end{equation}
where $\bR(T)\coloneqq\{\bv_T:~\bv_T(\bx)=A_T(\hat{x})\hbv_{\hT}(\hat{x}) \text{ with } \hbv_{\hT} \in \bP_{1}(\hT)+\hat{x} P_{1}(\hT) \text{ and }x=F_T(\hat{x})\}.$

Recall the canonical interpolation operator $\hat{\Pi}_{\hT}$ for the first-order Raviart–Thomas element, see, e.g.,~\cite{boffi2013mixed}, as
\begin{align}
\label{PihT1}
\int_{\hat{e}}(\hat{\Pi}_{\hT}\hbv_{\hT}\cdot\hat{\bn}_{\hat{e}})\hq\hds&=\int_{\hat{e}}(\hbv_{\hT}\cdot\hbn_{\hat{e}})\hq\hds \,  \text{ for all } \hat{e}\subset\partial\hat{T},~\hq\in P_{1}(\hat{e}),\\
\label{PihT2}
\int_{\hT}(\hat{\Pi}_{\hT}\hbv_{\hT})\cdot\hat{\boldsymbol{q}}\dx&=\int_{\hT}\hbv_{\hT}\cdot\hat{\boldsymbol{q}}\hdx \quad \, \text{ for all } \hat{\boldsymbol{q}} \in\bP_{0}(\hT).
\end{align}
For any $\bv_T\in H^1(T)$, let $\hbv_{\hT}(\hat{x})=A_T^{-1}(\hat{x})\bv_T(F_T(\hat{x}))$ and define the local interpolation operator $\Pi_T$ as  $\Pi_T\bv_T(\bx) = A_T(\hat{x})(\hat{\Pi}_{\hT}\hbv_{\hT})(\hat{x})$ with $x=F_T(\hat{x})$.     
For any $\bv\in\bV+\bVh$, one can define the global reconstruction operator $\Pi_h$ by $\Pi_h\bv|_T=\Pi_T(\bv|_T)$ for all $T\in\calTh$.    
The next lemma provides approximation properties of $\Pi_h$.
\begin{lemma}\label{PiVVhinWh}
For any $\bv\in\bV+\bVh$, it holds that $\Pi_h\bv\in\bR_h$.
Moreover, for all $T\in\calTh$ and $e\subset\partial T$, the following properties hold
\begin{align}
\label{ProPih::2}
& \int_e (\bv-\Pi_h\bv)\cdot\bn_e \, q\ds =0 \text{ for all } q(x)=\hat{q}(\hx)\in P_1(\hat{e}) \text{ with } x= F_T(\hx),\\
\label{ProPih::1}
& \int_T(\bv-\Pi_h\bv)\cdot( (DF_T^{-1})^\intercal \boldsymbol{q})\dx =0  \text{ for all } \boldsymbol{q} \in\bP_{0}(T), \\
\label{ProPih::3}
& |\Pi_h\bv-\bv|_{H^l(T)}\leq Ch_T^{m-l}\|\bv\|_{H^m(T)} \text{ with } m\in\{0,1,2 \}, l\in\{0,1\},l\leq m.
\end{align}
\end{lemma}

\begin{proof}
Let $\bv_T=\bv|_T$ for all $T \in \calTh$ and $\hbv_{\hT}(\hat{x})=A_T^{-1}(\hat{x})\bv_T(F_T(\hat{x}))$ with $x = F_T(\hx)$. 
By Lemma \ref{lem:pro-AT}, the definition of $\Pi_T$ and \eqref{PihT1} show
\begin{equation*}
\begin{aligned}
\int_{e} (\bv_T-\Pi_T \bv_T)\cdot\bn_e \, q\ds =
\int_{\hat{e}} (\hbv_{\hT}-\hat{\Pi}_{\hT} \hbv_{\hT})\cdot \hbn_{\hat{e}} \, \hq \hds = 0 
\end{aligned}    
\end{equation*}
for all $\hq \in P_1(\hat{e})$ and $q(x)=\hq(\hx)$ with $x=F_T(\hx)$, which proves \eqref{ProPih::2}. If $\bv \in \bV$, $\bv\cdot\bn_e$ is continuous across each interior edge and vanishes on $\partial\Omega_h$ for $\bv\in\bV$. This, \eqref{def:Rh}, and \eqref{ProPih::2} show $\Pi_h\bv\in\bR_h$ for $\bv \in \bV$. If $\bv \in \bVh$, by Theorem~\ref{the::VhEqXhDirSumPhih}, there exist $\bw_h\in\bWh$ and $\bphi_h\in\bPhi_h$ such that $\bv=\bw_h+\bphi_h$. Note that $\bw_h \in \bH_0(\Div;\Omega_h)$ and $\bphi_h$ vanishes at the two Gauss--Legendre points of each edge. This, Gauss--Legendre’s quadrature rule, \eqref{def:Rh}, and \eqref{ProPih::2} proves $\Pi_h\bv\in\bR_h$ for $\bv \in \bVh$.
Recall the definition of $A_T$ in \eqref{def: piola-AT} and $A_T^{-\intercal} = \det (DF_T) (DF_T^{-1} \circ F_T)^\intercal$. By \eqref{PihT2}, it holds that
$$
\int_T(\bv-\Pi_T\bv)\cdot( (DF_T^{-1})^\intercal \boldsymbol{q})\dx = \int_{\hT} A_T(\hbv-\hat{\Pi}_{\hT} \hbv)\cdot( A_T^{-\intercal} \hat{\boldsymbol{q}})\hdx = 0 
$$
for all $\hat{\boldsymbol{q}} \in \hat{\bP}_0(\hT)$ and $\boldsymbol{q}(x) = \hat{\boldsymbol{q}}(\hx)$ with $x=F_T(\hx)$, which proves \eqref{ProPih::1}. Finally, for any $m\in\{0,1,2\}$, the definition of $\Pi_T$ and Lemma \ref{lem::normH1XThXT} show
$$
\begin{aligned}
|\bv-\Pi_T\bv|_{H^l(T)} \leq Ch_T^{-l}\sum_{r=0}^l|\hbv - \hat{\Pi}_{\hT}\hbv|_{H^r(T)}
\leq Ch_T^{-l}|\hbv|_{H^m(T)}\leq Ch_T^{m-l}\|\bv\|_{H^{m}(T)},
\end{aligned}
$$
which completes the proof of \eqref{ProPih::3}.
\end{proof}

\begin{lemma}
\label{lem::DivPihCom}
For any $\bv_h \in\bVh$ and $q_h \in Q_h$, it holds that
\begin{equation}
\label{DivPihEqDiv}
\int_{\Omega_h} q_h \Div\Pi_h\bv_h\dx=\int_{\Omega_h} q_h \Div_h \bv_h \dx.    
\end{equation}
As a result, $\Pi_h(\bZ_h)\subset\bR_0$, where $\bR_0\coloneqq\{\bv_h \in\bR_h:~\Div \bv_h =0\}$.
\end{lemma}
\begin{proof}
Define $q_T=q_h|_T$ and $\bv_T=\bv_h|_T$ for all $T\in\calTh$.
Let $\hq_T(\hx) = q_
T(x)$ with $\hx = F_T^{-1}(x)$, which belongs to $P_1(\hT)$ by the definition of $Q_h$ in \eqref{defQinterpolate}. 
The definition of $\Pi_T$, Lemma~\ref{lem:pro-AT}, an integration by parts on $\hT$, and \eqref{PihT1}--\eqref{PihT2} show
\begin{equation*}
\int_{T} q_T \Div(\bv_T-\Pi_T\bv_T)\dx=\int_{\hT} \hq_T \Hat{\Div}(\hbv_T-\hat{\Pi}_{\hT}\hbv_T)\hdx=0.
\end{equation*}
Summing this over all elements $T \in \calTh$ proves \eqref{DivPihEqDiv}. Thus, for any $\bv_h \in\bZ_h$,  it holds that $\int_{\Omega_h}q_h \Div\Pi_h\bv_h\dx=0$  for all $q_h\in Q_h$. This and similar procedures as in Lemma \ref{lem:div-free} lead to $\Div\Pi_h\bv_h=0$, which implies $\Pi_h\bv_h\in\bR_0$. 
\end{proof}

Let $\hat{\pi}_{\hT}^0$ be the $L^2$-projection from $\bL^2(\hT)$ to $\bP_{0}(\hT)$. For any $\bv_T \in H^1(T)$, define 
\begin{equation}
\label{def:piT0}
\pi_{T}^{0}\bv_T(x)=(DF_T^{\intercal}(\hx))^{-1} \hat{\pi}_{\hT}^0 (DF_T^{\intercal}(\hx) \hbv_{\hT}(\hx))    
\end{equation}
for all $\bv_T(x)= \hbv_{\hT}(\hx)$ with $x = F_T(\hx)$. Lemma \ref{lem::NormInThT}, \eqref{Prop-FT:eq1}, the Bramble-Hilbert lemma, and Leibniz's rule give
\begin{equation*}
\begin{aligned}
\|\bv-\pi_{T}^{0}\bv\|_{L^2(T)} & \leq C h_T \|  (DF_T^{\intercal})^{-1} (DF_T^{\intercal} \hbv_{\hT}  - \hat{\pi}_{\hT}^0 (DF_T^{\intercal} \hbv_{\hT} )  )\|_{L^2(\hT)}   \\
& \leq  C h_T \| (DF_T^{\intercal} \circ F_T^{-1})^{-1} \|_{L^\infty(T)} \| DF_T^{\intercal} \hbv_{\hT}  - \hat{\pi}_{\hT}^0 (DF_T^{\intercal} \hbv_{\hT} ) \|_{L^2(\hT)} 
\\
& \leq  C h_T  h_T^{-1} |DF_T^{\intercal} \hbv_{\hT}|_{H^1(\hT)} \leq C\sum_{r=0}^{1}|DF_T|_{W^{1-r,\infty}(\hT)}|\hbv_{\hT}|_{H^{r}(\hT)}.
\end{aligned}
\end{equation*}
Further application of Lemma \ref{lem::NormInThT} and \eqref{Prop-FT:eq1} yields
\begin{equation}
\label{AppropiT0}
\|\bv-\pi_{T}^{0}\bv\|_{L^2(T)}  \leq C \sum_{r=0}^{1} h_T^{2-r} h_T^{r-1} \| \bv_T\|_{H^r(T)} \leq C h_T \| \bv_T\|_{H^1(T)}.
\end{equation}

\begin{lemma}
\label{lem:Pih3}
Let $\bv\in\bV+\bVh$. For any $\bu\in \bH^{3}(\Omega)$, the following estimate holds
\begin{align}
\label{DelatuUErrors}
\Big|\int_{\Omega_h}\Delta\bu\cdot\Pi_h\bv+\nabla\bu:\nabla_h \bv\dx\Big|&
\leq Ch^2\|\bu\|_{H^3(\Omega)}\|\nabla_h \bv\|_{L^2(\Omega_h)}.
\end{align}
\end{lemma}
\begin{proof}
Adding and subtracting the term $\int_{\Omega_h}\Delta\bu\cdot\bv\dx$, one can obtain
\begin{equation}
\label{DelatUEst}
\int_{\Omega_h}\Delta\bu\cdot\Pi_h\bv+\nabla\bu:\nabla_h \bv\dx = \int_{\Omega_h}\Delta\bu\cdot(\Pi_h\bv-\bv)+(\nabla\bu:\nabla_h \bv+\Delta\bu\cdot\bv)\dx.
\end{equation}
For the first term on the right-hand side, the combination of \eqref{ProPih::1}, \eqref{def:piT0}, \eqref{AppropiT0}, and \eqref{ProPih::3} leads to $|\int_T\Delta\bu\cdot(\Pi_h\bv-\bv)\dx|=|\int_T(\Delta\bu-\pi_T^0\Delta\bu)\cdot(\Pi_h\bv-\bv)\dx|\leq Ch^2_T\|\bu\|_{H^3(T)}\|\bv\|_{H^1(T)}$. 
Summing this over all elements $T\in\calTh$, the Poincar\'e inequality and Lemma \ref{col::H1Norm} yield
\begin{equation}\label{DelatUEst1}
\begin{aligned}
\Big|\int_{\Omega_h}\Delta\bu\cdot(\Pi_h\bv-\bv)\dx\Big|
&\leq Ch^2\|\bu\|_{H^3(\Omega)}\|\nabla_h \bv\|_{L^2(\Omega_h)}.
\end{aligned}    
\end{equation}
For the second term on the right-hand side, an integration by parts gives
$
\int_{\Omega_h}\nabla\bu:\nabla_h \bv\dx=-\int_{\Omega_h}\Delta\bu\cdot\bv\dx+\sum_{e\in\calEh}\int_e\nabla\bu:[\bv]\ds.
$
It follows from the consistency error estimate in Lemma \ref{lem::nonconErrors} that 
\begin{equation}
\label{DelatUEst2}
\Big|\int_{\Omega_h}\nabla\bu:\nabla_h \bv+\Delta\bu\cdot\bv\dx\Big|
\leq Ch^2\|\bu\|_{H^3(\Omega)}\|\nabla_h \bv\|_{L^2(\Omega_h)}.
\end{equation}
Substituting \eqref{DelatUEst1} and \eqref{DelatUEst2} into \eqref{DelatUEst} completes the proof.
\end{proof}

\subsection{Modified scheme and error estimates}
\label{modifySchemeError} 
Recall the operator $\Pi_h^Y$ introduced in Theorem \ref{the::PiY} and let $\bft_h=\Pi_h^Y\bft$ in the following. 
The modified discrete scheme reads: Find $(\buR,\pR)\in\bVh\times Q_h$ such that
\begin{subequations}
\label{eqn:FEMR}
\begin{alignat}{2}
\label{eqn:FEMR1}
\int_{\Omega_h} \nu \nabla_h \buR:\nabla_h \bv_h \dx- \int_{\Omega_h}  \pR  \Div_h \bv_h  \dx & = \int_{\Omega_h} \bft_h\cdot \Pi_h\bv_h \dx &,\\
\label{eqn:FEMR2}
\int_{\Omega_h} q_h \Div_h \buR  \dx & =0 &
\end{alignat}
\end{subequations}
for all $(\bv_h,q_h)\in\bVh\times Q_h$.

\begin{theorem}
\label{thm: robust}
Let $(\bu,p)$ be the solution of \eqref{eq::Stokes} and $(\buR,\pR)\in\bVh\times Q_h$ be the solution of \eqref{eqn:FEMR}. If $(\bu,p)$ belongs to $\bH^5(\Omega)\times H^3(\Omega)$, then it holds that
$$
\|\nabla \bu- \nabla_h \buR\|_{L^2(\Omega_h)}\leq Ch^2\|\bu\|_{H^5(\Omega)}.
$$
\end{theorem}
\begin{proof}
Similar procedures to derive \eqref{es1:eq1} in Theorem \ref{thm:es1} lead to
\begin{equation}
\label{uRHerror0}
\begin{aligned}
\nu\|\nabla \bu- \nabla_h \buR\|_{L^2(\Omega_h)}&\leq C\nu h^2\|\bu\|_{H^3(\Omega)}+\sup_{\bv_h \in\bZ_h\setminus\{0\}}\frac{a_h(\buR-\bu,\bv_h)}{\|\nabla_h \bv_h \|_{L^2(\Omega_h)}} 
\end{aligned}    
\end{equation}
with 
$$
a_h(\buR-\bu,\bv_h)
=\int_{\Omega_h}\bft_h\cdot\Pi_h\bv_h \dx -\int_{\Omega_h}\nu\nabla\bu:\nabla\bv_h \dx.
$$
Since $\bft=-\nu\Delta \bu+\nabla p$ in $\mathbb{R}^2$, Theorem \ref{the::PiY} implies $\bft_h=\Pi_h^Y\bft=-\nu\Pi_h^Y\Delta\bu+\nabla\Pi_h^\Sigma p$. For any $\bv_h \in\bZ_h$, an element-wise integration by parts and $\Pi_h\bv_h\in\bR_0$ shown in Lemma \ref{lem::DivPihCom} give
\begin{equation*}  
\int_{\Omega_h} \nabla\Pi_h^\Sigma p \cdot \Pi_h\bv_h \dx = -\int_{\Omega_h}\Pi_h^\Sigma p\Div\Pi_h\bv_h \dx + \sum_{e\in \mathcal{E}_h} \int_{e} \Pi_h^\Sigma p [\Pi_h\bv_h \cdot \bn_e] \ds =0.
\end{equation*}
This and adding and subtracting the term $-\nu \int_{\Omega_h} \Delta \bu \cdot \Pi_h \bv_h \dx$ lead to 
\begin{equation}
\label{uRHeRROR}
\begin{aligned}
a_h(\buR&-\bu,\bv_h)
= - \nu \int_{\Omega_h}\Pi_h^Y\Delta\bu \cdot\Pi_h\bv_h-\nu\nabla\bu:\nabla\bv_h \dx\\
&= -\nu\int_{\Omega_h}(\Pi_h^Y\Delta\bu-\Delta\bu)\cdot\Pi_h\bv_h+(\Delta\bu\cdot\Pi_h\bv_h+\nabla\bu:\nabla\bv_h)\dx.
\end{aligned}    
\end{equation}
For the first term on the right-hand side, the Cauchy-Schwarz inequality, \eqref{PiYEst}, \eqref{ProPih::3} with $m=l=0$ and Lemma \ref{col::H1Norm} show 
\begin{equation}
\label{PiYDuPihErr}
\Big|\int_{\Omega_h}(\Pi_h^Y\Delta\bu-\Delta\bu)\cdot\Pi_h\bv_h \Big|\leq C h^2\|\Delta\bu\|_{H^3(\Omega)}\|\nabla_h \bv_h\|_{L^2(\Omega_h)}.
\end{equation}
The combination of Lemma \ref{lem:Pih3} and \eqref{uRHerror0}--\eqref{PiYDuPihErr} completes the proof.
\end{proof}

\begin{remark}
\label{remark:uh0}
Similar to~\cite[Remark 3.7]{linke2016robust}, one can show: If $\bft_h= \nabla \psi$ with $\psi \in H^1(\Omega)$, then the solution $\bu_h^R$ of the modified problem \eqref{eqn:FEM} satisfies
$$
\nu \int_{\Omega_h} \nabla_h \bu_h^R :\nabla_h \bv_h \dx =  \int_{\Omega_h} \nabla \psi \cdot \Pi_h \bv_h \dx  =  - \int_{\Omega_h} \psi \Div  \Pi_h \bv_h \dx = 0
$$
for all $\bv_h  \in \bZ_h$,
which directly implies $\bu_h^R = \mathbf{0}$.  
\end{remark}

\section{Numerical Examples}
\label{sec::numerical_examples}

This section conducts two numerical experiments to validate the convergence, element-wise divergence-free property, and pressure-robustness of the modified scheme \eqref{eqn:FEMR}. 
For comparison, the numerical results of the standard scheme \eqref{eqn:FEM} are provided. In the following, let $\Omega=\{(x_1,x_2)^\intercal \in\mathbb{R}^2:~x_1^2+x_2^2<1\}$ and $\bft_h=\Pi_h^Y\bft$. For ease of presentation in the tables, this section uses $\|\bullet\|_0$ to denote $\|\bullet\|_{L^2(\Omega_h)}$.

\subsection{No flow problem}
Consider the Stokes problem \eqref{eq::Stokes} with $\nu=1$. The right-hand side of \eqref{eq::Stokes} is given by the gradient force $\bft=\nabla \psi$ with $\psi=2x_1^2(1-x_1)x_2(1-x_2)$.
The solution of \eqref{eq::Stokes} is given by
$$
\bu=\mathbf{0},\quad p=\psi+\frac{1}{12},
$$
where the constant is added to ensure $p\in L^2_0(\Omega)$.

Table~\ref{tab:no_flow_compare} presents the errors and convergence rates for the standard and modified schemes together with the $L^2$-norm of piecewise-divergence for the velocity. It can be seen that the convergence rates of the standard scheme are
$$
\|\bu-\bu_h\|_{0}=\mathcal{O}(h^3), \, \|\nabla \bu- \nabla_h \bu_h\|_{0}=\mathcal{O}(h^2), \, \|p-p_h\|_{0}=\mathcal{O}(h^2).
$$
which agree with the theoretical results stated in Theorem \ref{thm:es1}. Note that the velocity field of the standard scheme does not vanish since the fundamental invariance condition is violated, see Remark~\ref{remark:uh0}.
In contrast,  the modified scheme generates the
expected zero flow field, which aligns with Theorem~\ref{thm: robust}.
Moreover, even though both schemes yield divergence-free discrete velocities, the modified scheme enforces this constraint better at the numerical level.

\begin{table}[ht]
\centering
\setlength{\tabcolsep}{3pt} {
\begin{tabular}{@{} w{c}{0.6cm} *{3}{w{c}{2.2cm} w{c}{0.8cm}} w{c}{1.8cm} @{}}
\toprule
& \multicolumn{7}{c}{Standard scheme} \\
\cmidrule(lr){2-8}
$1/h$& {$\|\bu-\bu_h\|_{0}$} & {rate} & {$\|\nabla \bu- \nabla_h \bu_h\|_{0}$} & {rate} & {$\|p-p_h\|_{0}$} & {rate} & {$\|\Div_h \bu_h\|_{0}$}\\
\midrule 
4  & 5.164E-05 & ---  & 1.937E-03 & ---  & 1.034E-02 & ---  & 1.317E-18 \\
8  & 5.953E-06 & 3.12 & 4.158E-04 & 2.22 & 2.591E-03 & 2.00 & 6.329E-19 \\
16 & 7.043E-07 & 3.08 & 9.482E-05 & 2.13 & 6.450E-04 & 2.01 & 3.119E-19 \\
32 & 8.544E-08 & 3.04 & 2.262E-05 & 2.07 & 1.606E-04 & 2.01 & 1.654E-19 \\
& \multicolumn{7}{c}{Modified scheme} \\
\cmidrule(lr){2-8}
$1/h$& {$\|\bu-\bu_h^R\|_{0}$} & {rate} & {$\|\nabla \bu- \nabla_h \bu_h^R\|_{0}$} & {rate} & {$\|p-p_h^R\|_{0}$} & {rate} & {$\|\Div_h \bu_h^R\|_{0}$}\\
\midrule
4  & 2.352E-17 & --- & 4.380E-16 & --- & 8.707E-03 & --- & 2.575E-31 \\
8  & 1.491E-17 & --- & 3.418E-16 & --- & 2.160E-03 & 2.01 & 4.660E-31 \\
16 & 1.105E-17 & --- & 4.777E-16 & --- & 5.358E-04 & 2.01 & 1.416E-30 \\
32 & 1.854E-17 & --- & 9.147E-16 & --- & 1.332E-04 & 2.01 & 5.215E-30 \\
\bottomrule
\end{tabular}}
\caption{No flow problem, $\nu=1$.}
\label{tab:no_flow_compare}
\end{table}			

\subsection{Problem with flow field}
Consider the Stokes problem \eqref{eq::Stokes} with the  solution given by
$$
\bu=\begin{pmatrix}
    \partial_{x_2}\psi\\
    -\partial_{x_1}\psi
\end{pmatrix} \text{ where }
\psi=\frac{1}{100}x_1^2(1-x_1)^2x_2^2(1-x_2)^2, 
p=2x_1^2(1-x_1)x_2(1-x_2)+\frac{1}{12}.
$$
The right-hand side term $\bft$ is obtained by substituting this solution into \eqref{eq::Stokes}.

Table~\ref{tab:with_flow_nu1_comparison} and Table~\ref{tab:with_flow_nu1_7_comparison} demonstrate the errors and convergence rates for the standard and modified
schemes with $\nu=1$ and $\nu=10^{-7}$, receptively. Optimal convergence rates of the velocity and pressure are achieved for these two schemes. However, the errors of $\|\bu-\bu_h\|_0$ and $\|\nabla \bu-\nabla_h\bu_h\|_0$ in Table~\ref{tab:with_flow_nu1_comparison} are much larger compared to the corresponding results in Table~\ref{tab:with_flow_nu1_7_comparison}, by roughly a factor of $\nu^{-1}$. In contrast, the errors produced by the modified scheme remain identical for $\nu=1$ and $\nu=10^{-7}$, confirming its pressure-robustness. This observation is further verified by numerical results shown in Table~\ref{tab:with_flow_diff_nu_comparison}, which presents the errors for the standard and modified schemes as $\nu$ varies from $1$ to $10^{-9}$ on a fixed mesh with $1/h=16$. Notably, the modified scheme preserves the divergence-free property more robustly compared to the standard scheme. 
\begin{table}[ht]
\centering
\setlength{\tabcolsep}{3pt} {
\begin{tabular}{@{} w{c}{0.6cm} *{3}{w{c}{2.2cm} w{c}{0.8cm}} w{c}{1.8cm} @{}}
\toprule
& \multicolumn{7}{c}{Standard scheme} \\
\cmidrule(lr){2-8}
$1/h$& {$\|\bu-\bu_h\|_{0}$} & {rate} & {$\|\nabla \bu- \nabla_h \bu_h\|_{0}$} & {rate} & {$\|p-p_h\|_{0}$} & {rate} & {$\|\Div_h \bu_h\|_{0}$}\\
\midrule 
4 & 5.861E-05 & --- & 2.203E-03 & --- & 1.052E-02 & --- & 9.045E-15  \\
8 & 6.862E-06 & 3.09 & 4.865E-04 & 2.18 & 2.635E-03 & 2.00 & 6.548E-16   \\
16 & 8.203E-07 & 3.06 & 1.129E-04 & 2.11 & 6.559E-04 & 2.01 & 5.471E-16   \\
32 & 1.000E-07 & 3.04 & 2.717E-05 & 2.06 & 1.634E-04 & 2.01 & 1.099E-15    \\
& \multicolumn{7}{c}{Modified scheme} \\
\cmidrule(lr){2-8}
$1/h$& {$\|\bu-\bu_h^R\|_{0}$} & {rate} & {$\|\nabla \bu- \nabla_h \bu_h^R\|_{0}$} & {rate} & {$\|p-p_h^R\|_{0}$} & {rate} & {$\|\Div_h \bu_h^R\|_{0}$}\\
\midrule
4 & 2.900E-05 & --- & 1.055E-03 & --- & 8.755E-03 & --- & 9.043E-15  \\
8 & 3.529E-06 & 3.04 & 2.529E-04 & 2.06 & 2.172E-03 & 2.01 & 6.564E-16   \\
16 & 4.335E-07 & 3.02 & 6.126E-05 & 2.05 & 5.386E-04 & 2.01 & 5.500E-16   \\
32 & 5.366E-08 & 3.01 & 1.503E-05 & 2.03 & 1.339E-04 & 2.01 & 1.100E-15    \\
\bottomrule
\end{tabular}}
\caption{Problem with flow, $\nu=1$.}
\label{tab:with_flow_nu1_comparison}
\end{table}	

\begin{table}[ht]
\centering
\setlength{\tabcolsep}{3pt} {
\begin{tabular}{@{} w{c}{0.6cm} *{3}{w{c}{2.2cm} w{c}{0.8cm}} w{c}{1.8cm} @{}}
\toprule
& \multicolumn{7}{c}{Standard scheme} \\
\cmidrule(lr){2-8}
$1/h$& {$\|\bu-\bu_h\|_{0}$} & {rate} & {$\|\nabla \bu- \nabla_h \bu_h\|_{0}$} & {rate} & {$\|p-p_h\|_{0}$} & {rate} & {$\|\Div_h \bu_h\|_{0}$}\\
\midrule 
4 & 5.164E+02 & --- & 1.937E+04 & --- & 1.034E-02 & --- & 1.313E-11  \\
8 & 5.953E+01 & 3.12 & 4.158E+03 & 2.22 & 2.591E-03 & 2.00 & 6.284E-12  \\
16 & 7.043E+00 & 3.08 & 9.482E+02 & 2.13 & 6.450E-04 & 2.01 & 3.155E-12 \\
32 & 8.544E-01 & 3.04 & 2.262E+02 & 2.07 & 1.606E-04 & 2.01 & 1.650E-12  \\
& \multicolumn{7}{c}{Modified scheme} \\
\cmidrule(lr){2-8}
$1/h$& {$\|\bu-\bu_h^R\|_{0}$} & {rate} & {$\|\nabla \bu- \nabla_h \bu_h^R\|_{0}$} & {rate} & {$\|p-p_h^R\|_{0}$} & {rate} & {$\|\Div_h \bu_h^R\|_{0}$}\\
\midrule
4 & 2.900E-05 & --- & 1.055E-03 & --- & 8.707E-03 & --- & 9.044E-15 \\
8 & 3.529E-06 & 3.04 & 2.529E-04 & 2.06 & 2.160E-03 & 2.01 & 6.533E-16 \\
16 & 4.335E-07 & 3.02 & 6.126E-05 & 2.05 & 5.358E-04 & 2.01 & 5.508E-16 \\
32 & 5.366E-08 & 3.01 & 1.503E-05 & 2.03 & 1.332E-04 & 2.01 & 1.102E-15 \\
\bottomrule
\end{tabular}}
\caption{Problem with flow, $\nu=10^{-7}$.}
\label{tab:with_flow_nu1_7_comparison}
\end{table}	

\begin{table}[ht]
\centering
\setlength{\tabcolsep}{3pt} {
\begin{tabular}{@{} w{c}{0.6cm} *{4}{w{c}{2.5cm}} @{}}
\toprule
& \multicolumn{4}{c}{Standard scheme} \\
\cmidrule(lr){2-5}
$\nu$& {$\|\bu-\bu_h\|_{0}$}  & {$\|\nabla \bu- \nabla_h \bu_h\|_{0}$} & {$\|p-p_h\|_{0}$}  & {$\|\Div_h \bu_h\|_{0}$}\\
\midrule 
$10^{-0}$ & 8.203E-07 & 1.129E-04 & 6.559E-04 & 5.471E-16 \\ 
$10^{-1}$ & 7.047E-06 & 9.505E-04 & 6.459E-04 & 5.501E-16 \\ 
$10^{-2}$ & 7.042E-05 & 9.483E-03 & 6.451E-04 & 5.543E-16 \\ 
$10^{-3}$ & 7.043E-04 & 9.483E-02 & 6.450E-04 & 6.481E-16 \\ 
$10^{-4}$ & 7.043E-03 & 9.482E-01 & 6.450E-04 & 3.213E-15 \\ 
$10^{-5}$ & 7.043E-02 & 9.482E+00 & 6.450E-04 & 3.139E-14 \\ 
$10^{-6}$ & 7.043E-01 & 9.482E+01 & 6.450E-04 & 3.127E-13 \\ 
$10^{-7}$ & 7.043E+00 & 9.482E+02 & 6.450E-04 & 3.155E-12 \\ 
$10^{-8}$ & 7.043E+01 & 9.482E+03 & 6.450E-04 & 3.121E-11 \\ 
$10^{-9}$ & 7.043E+02 & 9.482E+04 & 6.450E-04 & 3.158E-10 \\
& \multicolumn{4}{c}{Modified scheme} \\
\cmidrule(lr){2-5}
$\nu$& {$\|\bu-\bu_h^R\|_{0}$}  & {$\|\nabla \bu- \nabla_h \bu_h^R\|_{0}$}  & {$\|p-p_h^R\|_{0}$}  & {$\|\Div_h \bu_h^R\|_{0}$}\\
\midrule
$10^{-0}$ & 4.335E-07 & 6.126E-05 & 5.386E-04 & 5.500E-16 \\ 
$10^{-1}$ & 4.335E-07 & 6.126E-05 & 5.358E-04 & 5.503E-16 \\ 
$10^{-2}$ & 4.335E-07 & 6.126E-05 & 5.358E-04 & 5.497E-16 \\ 
$10^{-3}$ & 4.335E-07 & 6.126E-05 & 5.358E-04 & 5.515E-16 \\ 
$10^{-4}$ & 4.335E-07 & 6.126E-05 & 5.358E-04 & 5.484E-16 \\ 
$10^{-5}$ & 4.335E-07 & 6.126E-05 & 5.358E-04 & 5.506E-16 \\ 
$10^{-6}$ & 4.335E-07 & 6.126E-05 & 5.358E-04 & 5.474E-16 \\ 
$10^{-7}$ & 4.335E-07 & 6.126E-05 & 5.358E-04 & 5.508E-16 \\ 
$10^{-8}$ & 4.335E-07 & 6.126E-05 & 5.358E-04 & 5.519E-16 \\ 
$10^{-9}$ & 4.335E-07 & 6.127E-05 & 5.358E-04 & 5.501E-16 \\ 
\bottomrule
\end{tabular}}
\caption{Problem with flow, mesh size $1/h=16$.}
\label{tab:with_flow_diff_nu_comparison}
\end{table}	

\section*{Acknowledgment}
\appendix

Thanks to Professor Dietmar Gallistl (University of Jena) and Professor Rui Ma (Beijing Institute of Technology) for their valuable comments on the initial draft of this paper.

\appendix
\section{Proofs of preliminary results}
\label{sec::appendix}

\subsection{Proof of Lemma \ref{lem:VTestimate1}}
\label{app::Prooflem:VTestimate1}

\begin{proof}
By the definitions of $\bW(T)$ and $\bPhi(T)$ in \eqref{def:WandPhi}, there exist $\hbw_{\hT} \in \hbW(\hT)$ and  $\hbphi_{\hT} \in \hbPhi(\hT)$ such that 
$\bw_T(\bx) = A_T(\hat{x}) \hbw_{\hT}(\hat{x})$ and $\bphi_T(\bx) = A_T(\hat{x}) \hbphi_{\hT}(\hat{x}) \text{ with } x = F_T(\hat{x})$, respectively.  
Let $\hbv_{\hT} = \hbw_{\hT} + \hbphi_{\hT}$, then $\bv_T = \bw_T + \bphi_T = A_T \hbv_{\hT}$. Since $A_T$ is invertible, it follows $\hbw_{\hT}(\hat{a}) = A_T^{-1}(\hat{a}) \bw_T(\ba) =0$ with $\hat{a} = F_T^{-1}(\ba)$. This gives 
\begin{equation}\label{InequalityNorm}
\|\hbw_{\hT}\|_{L^2(\hT)}+\|\hbphi_{\hT}\|_{L^2(\hT)}\leq C\|\hbv_{\hT}\|_{L^2(\hT)}.
\end{equation}
Assume that \eqref{InequalityNorm} is not valid. Then, there exist sequences $\hbw_{n,{\hT}}\in\hbW(\hT)$ with $\hbw_{n,{\hT}}(\hat{a})=0$, $\hbphi_{n,{\hT}}\in\hbPhi(\hT)$ and $\hbv_{n,\hT}=\hbw_{n,{\hT}}+\hbphi_{n,{\hT}}$ such that 
$$
\|\hbw_{n,{\hT}}\|_{L^2(\hT)}+\|\hbphi_{n,{\hT}}\|_{L^2(\hT)}=1 \text{ and }\lim_{n \rightarrow \infty} \|\hbv_{n, \hT} \|_{L^2(\hT)} = 0.
$$
Since $\hbW(\hT)$ and $\hbPhi(\hT)$ are finite-dimensional spaces, there exist subsequences $\hbw_{n_k,{\hT}}$ and $\hbphi_{n_k,{\hT}}$ converge to $\hbw_{*,\hT}$ and $\hbphi_{*,\hT}$, respectively. Let $\hbv_{*,\hT}=\hbw_{*,\hT}+\hbphi_{*,\hT}$. Then 

$$
\hbw_{*,\hT}(\hat{a})=0,  \, \|\hbw_{*,\hT}\|_{L^2(\hT)}+\|\hbphi_{*,\hT}\|_{L^2(\hT)}=1 \text{  and  }
\|\hbv_{*,\hT}\|_{L^2(\hT)}=0.
$$
This gives $\hbv_{*,\hT}=0$, which yields $\hbphi_{*,\hT}(\hat{a}) = \hbv_{*,\hT}(\hat{a})-\hbw_{*,\hT}(\hat{a})=0$.
\eqref{def:hbPhi} implies $\hbphi_{*,\hT}=0$ and thus $\hbw_{*,\hT}=0$, which contradicts the fact that $\|\hbw_{*,\hT}\|_{L^2(\hat{T})} + \|\hbphi_{*,\hT}\|_{L^2(\hat{T})} = 1$. 
Therefore, inequality \eqref{InequalityNorm} holds. This and Lemma \ref{lem::normH1XThXT} lead to
$$
\|\bw_T\|_{L^2(T)} + \|\bphi_T\|_{L^2(T)}
\leq C(\|\hbw_{\hT}\|_{L^2(\hT)}+\|\hbphi_
{\hT}\|_{L^2(\hT)})\leq C\|\hbv_{\hT}\|_{L^2(\hT)} \leq C\|\bv_T\|_{L^2(T)}.
$$

If $\bw_T$ vanishes at the two endpoints and the midpoint of an edge $e$ of $T$, then $\hbw_{\hT}(\hx) = A_T^{-1} \bw(x)$ with $x=F_T(\hx)$ vanishes the two endpoints and the midpoint of an edge $\hat{e}$ of $\hT$ with $e = F_T(\hat{e})$.
Since $\hbw_{\hT}$ is a quadratic polynomial on $\hat{e}$, this implies $\hbw_{\hT}|_{\hat{e}}=0$, thus $\bw_T|_{e}=0$. 
Claim that $\hbv_{\hT}\to\|\hat{\nabla} \hbv_{\hT} \|_{L^2(\hT)}$ is a norm. If $\|\hat{\nabla}\hbv_{\hT}\|_{L^2(\hT)}=0$, then $\hbv_{\hT}\equiv c$. Since $ \hbphi_{\hT}|_{\hat{e}} = \hbv_{\hT}|_{\hat{e}} - \hbw_{\hT}|_{\hat{e}} \equiv c$ and $\hbphi_{\hT}$ vanishes on the two Gauss--Legendre points on $\hat{e}$, it follows $c =0$. Combining this result with the equivalence of norms and Lemma~\ref{lem::normH1XThXT} can obtain
$$
\|\bv_T\|_{L^2(T)}\leq C\|\hbv_{\hT}\|_{L^2(\hT)}\leq C\|\hat{\nabla}\hbv_T\|_{L^2(\hT)}\leq Ch_T\big(\|\bv_T\|_{L^2(T)}+\|\nabla\bv_T\|_{L^2(T)}\big).
$$
Therefore, for sufficiently small $h_T$ satisfying $Ch_T\leq \frac{1}{2}$,  inequality \eqref{bvL2ConByNabv} holds.
\end{proof}

\subsection{Proof of Lemma \ref{lem::EstimateNoncon}}
\label{app::Prooflem::EstimateNoncon}

\begin{proof}
Let $\hat{\varphi}(\hx)=\varphi(x)$ and $\hat{v}(\hx) = v(x)$ with $\hx = F_T^{-1}(x)$. Recall $\Je(\hat{x})=\|A_T^{-\intercal}(\hat{x})\hat \bn_{\hat{e}}\|_{\mathbb{R}^2}$ for $\hx \in  \hT$. It follows from \eqref{pro-Je} and the definition of $\Pe$ that
\begin{equation}\label{ErrorPe1}
\int_e\varphi(v-\Pe v)\ds=\int_{\hat{e}}\hat{\varphi}\big(\Je\hat{v}-\Phe(\Je\hat{v})\big)\hds.
\end{equation}
Fix $\hat{v}\in H^{m+1}(\hT)$ and consider the linear functional
$
\hat{\varphi}\to\int_{\hat{e}}\hat{\varphi}\big(\Je\hat{v}-\Phe(\Je\hat{v})\big)\hds
$
on $H^1(\hT)$. It is continuous with norm bounded by $\|\Je\hat{v}-\Phe(\Je\hat{v})\|_{L^2(\hat{e})}$ and vanishes on constants $P_0$ by~\eqref{defPhe}. Applying the Bramble-Hilbert lemma gives
\begin{equation}\label{ErrorPe2}
\left|\int_{\hat{e}}\hat{\varphi}\big(\Je\hat{v}-\Phe(\Je\hat{v})\big)\hds\right|\leq C |\hat{\varphi}|_{H^1(\hT)}\|\Je\hat{v}-\Phe(\Je\hat{v})\|_{L^2(\hT)}.
\end{equation}
Since $\Phe \hat{v}=\hat{v}$ for all $\hat{v}\in P_1(\hat{e})$, the Bramble-Hilbert lemma implies
\begin{equation}
\label{ErrorPe3}
\|\Je\hat{v}-\Phe(\Je\hat{v})\|_{L^2(\hT)}\leq C|\Je\hat{v}|_{H^{m+1}(\hT)}\leq C\sum_{j=0}^{m+1}|\Je|_{W^{j,\infty}(\hT)}|\hat{v}|_{H^{m+1-j}(\hT)}.
\end{equation}
Based on the estimate of $A_T$ in \eqref{ATBounds}, simple calculations show
\begin{equation}
\label{JeBounds}
|\Je|_{W^{j,\infty}(\hT)}\leq Ch_T^{j+1} \text{ with }  j=0,1,2. 
\end{equation}
The equation \eqref{ErrorPe1}, combined with \eqref{ErrorPe2}--\eqref{JeBounds}, shows
\begin{equation*}
\left|\int_e\varphi(v-\Pe v)\ds\right|\leq C|\hat{\varphi}|_{H^1(\hT)}\sum_{j=0}^{m+1}h_T^{j+1}|\hat{v}|_{H^{m+1-j}(\hT)}.
\end{equation*}
Using the results $|\hat{v}|_{H^j(\hT)}\leq Ch_T^{j-1}\|v\|_{H^j(T)}$ provided by Lemma \ref{lem::NormInThT}, one can obtain
$$
\begin{aligned}
\left|\int_e\varphi(v-\Pe v)\ds\right|&\leq C|\varphi|_{H^1(T)}\sum_{j=0}^{m+1}h_T^{j+1}h_T^{m+1-j-1}\|v\|_{H^{m+1-j}(T)}\\
&\leq Ch_T^{m+1}|\varphi|_{H^1(T)}\|v\|_{H^{m+1}(T)},
\end{aligned}
$$
which completes the proof.
\end{proof}

\bibliographystyle{abbrv}
\bibliography{evp}

\end{document}

%% file: abbr.tex
\newcommand{\bu}{\boldsymbol{u}}
\newcommand{\bH}{\boldsymbol{H}}
\newcommand{\bv}{\boldsymbol{v}}
\newcommand{\bV}{\boldsymbol{V}}
\newcommand{\bVISO}{\bV^{iso}_h}
\newcommand{\bx}{x}

\newcommand{\bP}{\boldsymbol{P}}
\newcommand{\bPhi}{\boldsymbol{\Phi}}
\newcommand{\bphi}{\boldsymbol{\phi}}
\newcommand{\bsig}{\boldsymbol{\sigma}}
\newcommand{\bw}{\boldsymbol{w}}

\newcommand{\bs}{\boldsymbol{s}}
\newcommand{\bn}{\boldsymbol{n}}
\newcommand{\hbn}{\hat{\bn}}
\newcommand{\bL}{\boldsymbol{L}}
\newcommand{\bZ}{\boldsymbol{Z}}
\newcommand{\bG}{\boldsymbol{G}}
\newcommand{\bft}{\boldsymbol{f}}
\newcommand{\ba}{a}

\newcommand{\bb}{\boldsymbol{b}}

\newcommand{\bI}{\boldsymbol{I}}
\newcommand{\bE}{\boldsymbol{E}}
\newcommand{\bW}{\boldsymbol{W}}
\newcommand{\bR}{\boldsymbol{R}}
\newcommand{\bY}{\boldsymbol{Y}}
\newcommand{\buR}{\boldsymbol{u}_{h}^{R}}
\newcommand{\pR}{p_{h}^R}
\newcommand{\bIT}{\bI^W_T}

\newcommand{\bWh}{\bW_{\!\!h}}
\newcommand{\bVh}{\bV_{\!\!h}}
\newcommand{\Gh}{G_h}
\newcommand{\dx}{\,\rd x}
\newcommand{\hdx}{\,\rd \hat{x}}
\newcommand{\ds}{\,\rd s}
\newcommand{\hds}{\,\rd\hat{s}}
\newcommand{\Div}{\operatorname{div}}
\newcommand{\hDiv}{\hat{\operatorname{div}}}

\newcommand{\tDiv}{\tilde{\operatorname{div}}}
\newcommand{\tnabla}{\tilde{\nabla}}
\newcommand{\Eh}{\boldsymbol{E}_h}

\newcommand{\Je}{J_e}
\newcommand{\Pe}{\pi_e^1}
\newcommand{\Phe}{\pi_{\hat{e}}^1}

\newcommand{\calTh}{\mathcal{T}_h}
\newcommand{\calEh}{\mathcal{E}_h}

\newcommand{\NT}{\mathcal{N}_T}
\newcommand{\GT}{\mathcal{G}_T}

\newcommand{\tT}{\tilde{T}}
\newcommand{\tq}{\tilde{q}}
\newcommand{\tQ}{\tilde{Q}}
\newcommand{\tbv}{\tilde{\bv}}
\newcommand{\tbw}{\tilde{\bw}}
\newcommand{\tbV}{\tilde{\bV}}
\newcommand{\tbx}{\tilde{\bx}}

\newcommand{\tbW}{\tilde{\bW}}
\newcommand{\tbPhi}{\tilde{\bPhi}}
\newcommand{\tbphi}{\tilde{\bphi}}
\newcommand{\tOmega}{\tilde{\Omega}}
\newcommand{\tcalTh}{\tilde{\mathcal{T}}_h}
\newcommand{\tba}{\tilde{a}}

\newcommand{\tNT}{\mathcal{N}_{\tT}}
\newcommand{\tGT}{\mathcal{G}_{\tT}}

\newcommand{\hx}{\hat{x}}
\newcommand{\hT}{\hat{T}}
\newcommand{\hq}{\hat{q}}
\newcommand{\hQ}{\hat{Q}}
\newcommand{\hbv}{\hat{\bv}}

\newcommand{\hbPhi}{\hat{\bPhi}}
\newcommand{\hbphi}{\hat{\bphi}}

\newcommand{\hNT}{\mathcal{N}_{\hT}}
\newcommand{\hGT}{\mathcal{G}_{\hT}}
\newcommand{\hbw}{\hat{\bw}}
\newcommand{\hbW}{\hat{\bW}}

\newcommand{\rd}{\mathrm{d}}